\documentclass[11pt,Letterpaper,oneside,reqno]{amsart}
\usepackage{amssymb, mathdots, url,graphicx,amscd, amsthm,fullpage}
\numberwithin{equation}{section}
\usepackage[usenames]{color}
\usepackage{enumerate}
\usepackage{paralist}
\usepackage{xcolor}
\usepackage{cases}
\usepackage[style=alphabetic,backend=biber]{biblatex}
\usepackage[margin=1in]{geometry}
\usepackage[hyperfootnotes=false, colorlinks, linkcolor={blue}, citecolor={magenta}, filecolor={blue}, urlcolor={blue}, plainpages=false, pdfpagelabels]{hyperref}

\def\cB{\mathcal{B}}
\def\cC{\mathcal{C}}
\def\cD{\mathcal{D}}

\def\cF{\mathcal{F}}

\def\cH{\mathcal{H}}

\def\cL{\mathcal{L}}

\def\cS{\mathcal{S}}

\def\cW{\mathcal{W}}

\def\AA{\mathbb{A}}

\def\CC{\mathbb{C}}

\def\GG{\mathbb{G}}

\def\QQ{\mathbb{Q}}
\def\RR{\mathbb{R}}

\def\VV{\mathbb{V}}
\def\WW{\mathbb{W}}
\def\XX{\mathbb{X}}
\def\YY{\mathbb{Y}}
\def\ZZ{\mathbb{Z}}

\newcommand\bbC{\mathbb{C}}

\newcommand\bfG{\mathbf{G}}
\newcommand\bfH{\mathbf{H}}

\newcommand\bfM{\mathbf{M}}
\newcommand\bfN{\mathbf{N}}

\newcommand\bfP{\mathbf{P}}

\newcommand\bfU{\mathbf{U}}
\newcommand\bfV{\mathbf{V}}
\newcommand\bfW{\mathbf{W}}

\newcommand\bfY{\mathbf{Y}}
\newcommand\bfZ{\mathbf{Z}}
\newtheorem{theorem}{Theorem}[subsection]
\newtheorem{claim}[theorem]{Claim}

\def\fg{\mathfrak{g}}

\def\fk{\mathfrak{k}}

\def\fm{\mathfrak{m}}

\def\fn{\mathfrak{n}}

\def\fp{\mathfrak{p}}

\def\Ad{\mathrm{Ad}}

\def\diag{\mathrm{diag}}
\def\dim{\mathrm{dim}\hspace{0.1cm}}

\def\End{\mathrm{End}}

\def\Hom{\mathrm{Hom}}

\def\Gal{\mathrm{Gal}}
\def\GL{\mathrm{GL}}

\def\Lie{\mathrm{Lie}}
\def\Ind{\mathrm{Ind}}

\def\Im{\mathrm{Im}}

\def\Res{\mathrm{Res}}

\def\SL{\mathrm{SL}}
\def\SU{\mathrm{SU}}

\def\Sym{\mathrm{Sym}}
\def\Sp{\mathrm{Sp}}

\def\Stab{\mathrm{Stab}}

\def\tr{\mathrm{tr}}

\def\U{\mathrm{U}}

\def\linspan{\textnormal{-span}}

\renewcommand{\tilde}[1]{\widetilde{#1}}
\renewcommand{\hat}[1]{\widehat{#1}}

\def\fin{\mathrm{fin}}

\usepackage{paralist}
\usepackage{tikz}
\usepackage{tikz-cd}
\usepackage{enumitem} 
\setlist[enumerate]{wide = 0pt, leftmargin=*}

\usepackage{mathtools}

\newtheorem{thm}{Theorem}[section]
\newtheorem{cor}[thm]{Corollary}
\newtheorem{prop}[thm]{Proposition}
\newtheorem{lemma}[thm]{Lemma}
\theoremstyle{definition}
\newtheorem{definition}[thm]{Definition}
\newtheorem{remark}[thm]{Remark}

\begin{document}

\title{Fourier Coefficients and Algebraic Cusp Forms on $\U(2,n)$}

\author{Anton Hilado}
\address{Department of Mathematics and Statistics, University of Vermont, Burlington, VT 05405, USA}
\email{anton.hilado@uvm.edu}

\author{Finn McGlade}
\address{Department of Mathematics, The University of Oklahoma, Norman, OK 73019, USA}
\email{finn.mcglade@ou.edu}

\author{Pan Yan}
\address{Department of Mathematics, The University of Arizona, Tucson, AZ 85721, USA}
\email{panyan@arizona.edu}


\date{\today}


\begin{abstract} 
We establish a theory of scalar Fourier coefficients for a class of non-holomorphic, automorphic forms on the unitary group $\mathrm{U}(2,n)$. By studying the theta lifts from holomorphic modular forms on $\mathrm{U}(1,1)$, we apply this theory to obtain examples of non-holomorphic cusp forms on $\U(2,n)$ whose Fourier coefficients are algebraic numbers. 
\end{abstract}

\maketitle

\goodbreak 

\tableofcontents

\goodbreak

\section{Introduction}
\label{section-introduction}
The Fourier coefficients of modular forms have served as a central object of study in number theory since the 19th century. Jacobi studied the Fourier coefficients of theta series in connection to the theory of quadratic forms \cite{Jacobi69}. Later, Siegel developed the Fourier expansion of holomorphic modular forms associated to the group of $2n$-by-$2n$ real symplectic matrices $\Sp(n)$. More recently, a certain class of \textit{quaternionic}\footnote{We use the word \textit{quaternionic} in the sense of \cite{GW96}.} modular forms on the split real Lie group $\mathrm{G}_2$ has been shown to have Fourier coefficients of arithmetic interest, see for example \cite{GGS02,LesliePollack22,pollack2023exceptional}. \\
\indent Let $n\in \ZZ_{\geq 1}$. The overarching goal of this paper is to study an analogous theory of Fourier coefficients for a class of modular forms on the unitary group $\U(2,n)$. These quaternionic modular forms on $\U(2,n)$ may be defined using the Schmid differentials $D_{\ell}^{\pm}$ associated to certain (continued) discrete series representations of $\U(2,n)$; for precise definitions of these operators, see Section~\ref{subsection-quaternionic-modular-forms-on-G}. Throughout this introduction we present our results using semi-classical notation, so that a weight $\ell\in \ZZ_{\geq 1}$ modular form $\varphi$ is, in particular, a function  on $\U(2,n)$ such that $D_{\ell}^{\pm}\varphi=0$.
\\ \indent Our first main theorem (Theorem \ref{thm1-intro}) is a multiplicity-at-most-one result for the generalized Whittaker spaces of certain quaternionic discrete series representation of $\U(2,n)$. In addition to giving a bound on the dimension of these generalized Whittaker space, Theorem \ref{thm1-intro} gives explicit formulas for the generalized Whittaker functions associated to the representation of $\U(2,n)$ of minimal $K$-type $\VV_{\ell}=\Sym^{\ell}\VV$. Here $\VV$ is a particular representation of a maximal compact subgroup $K_{\infty}$ of $\U(2,n)$. The results allow us to associate a set of scalar \textit{Fourier coefficients} $\{a_{\varphi}(T)\}_{T\geq 0}$ to a quaternionic modular form $\varphi$ on $\U(2,n)$.  These coefficients $a_{\varphi}(T)$ are indexed by vectors $T$ in the homogeneous cone of a rational hermitian space  $\bfV_0$ of signature $(1,n-1)$. \\
 \indent In our second main theorem (Theorem \ref{thm3-intro}), we construct quaternionic cusps forms $\varphi$ whose Fourier coefficients $a_{\varphi}(T)$ are algebraic numbers and not all equal to zero. The construction proceeds by theta lifting  of holomorphic cusp forms on $\U(1,1)$ using a specific choice of archimedean test data. As a byproduct of our analysis, we obtain formulas for the quaternionic cusp forms on $\U(2,n)$ which are obtained by considering theta lifts of holomorphic Poincar\'e series from $\U(1,1)$. \\
 \indent We now set up the notation necessary to state Theorem \ref{thm1-intro}. Suppose $V$ is a non-degenerate Hermitian space over $\CC$ of signature $(2,n)$.
 Let $P=M N$ denote the parabolic subgroup of $\U(2,n)$ stabilizing a fixed isotropic line $U\subseteq V$. The Levi factor $M\leq P$ is identified as $M= \U(V_0)\times \CC^{\times}$
   where $V_0$ is a fixed $n$-dimensional subspace of $V$ which has signature $(1,n-1)$ and is orthogonal to $U$. The unipotent radical $N$ is non-abelian with a one-dimensional center $Z=[N,N]$. We have an identification $V_0\xrightarrow{\sim} N^{\mathrm{ab}}$ (see \S \ref{The-Lie-algebras}) which we denote $w\mapsto \exp(w)$. 
The maximal compact $K_{\infty}$ stabilizes an orthogonal decomposition $V=V^+_2\oplus V^-_n$ where $V_2^+$ (resp. $V_n^-$) is a definite subspace of dimension $2$ (resp. $n$). Fix a unit vector $u_2\in V_0\cap V_2^+$ and let $u_1\in V_2^+$ be a second unit vector satisfying $\langle u_1, u_2\rangle=0$. Then $\{u_1^{\ell-v}u_2^{\ell+v} \colon v=-\ell, \ldots, \ell\}$ defines a basis for $\VV_{\ell}=(\Sym^{2\ell}V_2^+\otimes \det_{\U(2)}^{-\ell})\boxtimes \mathbf{1}$. Here $\mathbf{1}$ denotes the trivial representation of $\U(V_n^{-})$. Set $\beta_T\colon M\to \CC$ by $ \beta_T(h,z)=\frac{4\pi}{\sqrt{2}}\langle u_2, zT\cdot h\rangle$ where $h\in \U(V_0)$ and $z\in \CC^{\times}$ so that $(h,z)\in M=\U(V_0)\times \CC^{\times}$.
   \begin{thm}[Theorem~\ref{Theorem-Multiplicity-At-Most-1}]
\label{thm1-intro}
Fix $T\in V_0$ non-zero. Write 
$
\cC_{N,T}^{\mathrm{md}}(\U(2,n), \VV_{\ell})_{K_{\infty}}^{\cD_{\ell}=0}
$ to denote the space of smooth functions $\cW_{T}\colon\U(2,n)\to \VV_{\ell}$ satisfying:
\begin{compactenum}[(i)]
\item $\cW_T$ is of moderate growth. 
\item If $k\in K_{\infty}$ and $g\in \U(2,n)$, then $\cW_T(gk)=\cW_T(g)\cdot k$.
\item The functions $\cD_{\ell}^+\cW_{\chi}$ and $\cD_{\ell}^-\cW_{\chi}$ vanish identically on $\U(2,n)$. 
\item If $w\in V_0$, $u\in Z$, and $g\in \U(2,n)$ then 
$
\cW_T(\exp(w)ug)=e^{-2\pi i \Im(\langle T,w\rangle)}\cW_T(g).
$
\end{compactenum}
We have
\begin{equation}
\label{main-inequality}
\dim_{\CC}\left(\cC_{N,T}^{\mathrm{md}}(\U(2,n), \VV_{\ell})_{K_{\infty}}^{\cD_{\ell}=0}\right)
=
\begin{cases} 1, &\hbox{if $\langle T,T\rangle \geq 0$}, \\
0, &\hbox{if $\langle T,T\rangle< 0$.}
\end{cases}
\end{equation}
If $\langle T,T\rangle \geq 0$, then there exists a unique function $\cW_T\in \cC_{N,T}^{\mathrm{md}}(\U(2,n), \VV_{\ell})_{K_{\infty}}^{\cD_{\ell}=0}$ satisfying
\begin{equation*}
\small \cW_{T}(h,z):=\sum_{-\ell\leq v\leq \ell}|z|^{2\ell+2}\left(\frac{|\beta_T(h,z)|}{\beta_T(h,z)}\right)^vK_v\left(|\beta_T(h,z)|\right)\frac{u_1^{\ell-v}u_2^{\ell+v}}{(\ell-v)!(\ell+v)!}
\end{equation*}
for all $(h,z)\in M=\U(V_0)\times \CC^{\times}$.
Here $K_v$ is the Bessel function $K_v(x)=\frac{1}{2}\int_0^{\infty}t^{v-1}e^{-x(t+t^{-1})}dt$.
\end{thm}
As a corollary to Theorem \ref{thm1-intro} we obtain a refined Fourier expansion for modular forms on $\U(2,n)$. 
\begin{cor}[Corollary~\ref{cor1}]
\label{corollary-intro1}
Suppose $\varphi$ is a modular form on $\U(2,n)$ satisfying $\cD_{\ell}^\pm \varphi \equiv 0$ and let $\varphi_N$ (resp. $\varphi_Z$) denote the constant term of $\varphi$ along $N$ (resp. $Z$). Then there exists a set of Fourier coefficients $\{a_{\varphi}(T)\in \CC\}_{T\in {V}_0}$ such that if $g\in \U(2,n)$ then
     \begin{equation}
     \label{eqn-intro-FE-QMF}
     \varphi_Z(g)=\varphi_N(g)+\sum_{T\in V_0-\{0\}\colon \langle T,T\rangle\geq 0 }a_{\varphi}(T)\cW_{T}(g). 
     \end{equation}
     Moreover, if $\varphi$ is cuspidal then \eqref{eqn-intro-FE-QMF} takes the form 
     \begin{equation}
         \label{eqn-intro-FE-CuspidalQMF}
         \varphi_Z(g)=\sum_{T\in V_0\colon \langle T,T\rangle> 0 }a_{\varphi}(T)\cW_{T}(g). 
     \end{equation}
\end{cor}
Our proof of Theorem \ref{thm1-intro} is adapted from \cite[Theorem 1.2.1]{Pollack21a} where the analogous multiplicity-at-most-one statement is obtained for modular forms on the quaternionic linear groups in Dynkin types $G_2, F_4, E_6, E_7$, $E_8$, $B_m$, and $D_{m+1}$ for $m\geq 3$. It relies on an analysis of the equations $D_{\ell}^{\pm}\varphi\equiv 0$ which is essentially carried out for the group $\SU(2,1)$ in \cite[Theorem 4.5]{KT95}. An analysis of the system $D_{\ell}^{\pm}\varphi\equiv 0$  in the case of $\SU(2,2)$ is made in \cite[\S 5 and \S 6]{Yamashita1991II}. Theorem \ref{thm1-intro} is closely related to a statement in representation theory. Suppose $\ell\geq n$ or $1\leq \lfloor \frac{n-1}{2}\rfloor\leq n<\ell$, and let $\Pi_{\ell}$ denote the $\U(2,n)$-representations with minimal $K$-type $\VV_{\ell}$ constructed in \cite{GW96}. The functions $\cW_T$ may occur as (generalized) Whittaker functions associated to $\Pi_{\ell}$. When $\Pi_{\ell}$ is discrete series and either $\langle T,T\rangle >0$ or $\langle T,T\rangle <0$, the formula \eqref{main-inequality} partially implies the type $A$ cases of a multiplicity-one result due to Wallach \cite[Theorem 16]{Wallach03}. Furthermore, there is a sense in which Theorem \ref{thm1-intro} extends Wallach's result to give a multiplicity-at-most-one statement in the cases when $\Pi_{\ell}$ is continued discrete series or when $T$ is non-zero and isotropic. We refer the reader to Remark \ref{rmk2-qmfs} for further details regarding the relationship between Theorem \ref{thm1-intro} and \cite[Theorem 16]{Wallach03}
\\ 
\indent 
We now move on to discuss Theorem \ref{thm2-intro} which gives examples of quaternionic modular forms on $\U(2,n)$. Suppose $W$ is a complex skew hermitian space of signature $(1,1)$ and let $\U(1,1)$ denote the isometry group of $W$. We take an isotropic basis $\{w_+,w_-\}$ of $W$ such that $\langle w_+,w_-\rangle_W=1$  where $\langle \,,\,\rangle_W$ denotes the skew-hermitian form on $W$. Writing elements of $\U(1,1)$ as matrices relative to the basis $\{w_+, w_-\}$ gives an identification $\U(1,1)=\{zg\colon z\in \CC^1, g\in \SL_2(\RR)\}$. Hence $\U(1,1)$ acts on the upper half plane $\cH_{1,1}=\{z\in \CC \colon \mathrm{Im}(z)>0\}$ by fractional linear transformations. Given an cuspidal automorphic function $\xi$ on $\U(1,1)$, the theory of the  Weil representation can be used to define certain theta liftings $\theta_{\psi,\chi}(\xi,\phi)$ of $\xi$ to the space of automorphic forms on $\U(2,n)$. In more detail, the Cartesian product $\U(1,1)\times \U(2,n)$ supports a special family of automorphic forms $\{\theta_{\psi,\chi}(\phi)\}$ known as \textit{theta series}. These theta series are indexed by triples $(\psi,\phi,\chi)$ consisting of Hecke character $\chi$, a finite adelic Schwartz function $\phi$, and an additive character $\psi$. We refer the reader to subsection \ref{subsection-The-metaplectic-group} for details. The lifting $\theta_{\psi,\chi}(\xi,\phi)$ is defined as the Petersson inner product of $\xi$ and $\theta_{\psi,\chi}(\phi)$. Theorem \ref{thm2-intro} describes the modular forms on $\U(2,n)$ which arise as theta lifts from $\U(1,1)$. 
\begin{thm}[Theorem~\ref{prop-relation-of-Fourier-coefficients}]
\label{thm2-intro}
Suppose $\Gamma\leq \U(1,1)$ is an arithmetic subgroup such that $\Gamma\backslash \cH_{1,1}$ is non-compact.
Suppose $\ell>n+1$ and let $\xi_f\colon \Gamma\backslash \U(1,1)\to \CC$ be the automorphic function associated to a holomorphic weight $2\ell+2-n$ modular form $f(\tau)=\sum_{t\in \QQ_{>0}}b_f(t)e^{2\pi i \tau t}$ on $\cH_{1,1}$. Assume $\xi_f(zg)=z^{n+2}\xi_f(g)$ for all $g\in \SL_2(\RR)$ and $z\in \CC^1$.  
\begin{compactenum}[(a)]
    \item There exists a triple $(\psi_0,\phi_0,\chi_0)$ such that the theta lift $\theta(\xi_f,\phi_0):=\theta_{\psi_0,\chi_0}(\xi_f,\phi_0)$ is a non-zero weight $\ell$ cuspidal quaternionic modular form on $\U(2,n)$.
    \item The constant term $\theta(\xi_f,\phi_0)_Z$ is non-zero and has Fourier expansion 
$$
\theta(\xi_f, \phi_0)_Z(g)=\sum_{T\in V_0\colon \langle T,T\rangle> 0 }a_{\theta(\xi_f, \phi_0)}(T)\cW_{T}(g)
$$
where each $a_{\theta(\overline{\xi}_f, \phi_0)}(T)$ is a finite linear combination of the Fourier coefficients $b_f(t)$.
\end{compactenum}
\end{thm}
Statement (a) of Theorem \ref{thm2-intro} can be deduced from the literature. For example, 
a duality theorem of R. Howe \cite{Howe89} together with the explicit determination of the local archimedean theta correspondence given in \cite{Li1990} implies the existence of $(\psi_0,\phi_0,\chi_0)$ such that $\theta(\xi_f, \phi_0)$ is a quaternionic modular form. For $n\geq 2$, the non-vanishing of $\theta(\xi_f, \phi_0)$ follows by an application of the tower property for unitary dual pairs (see for example \cite{Wu2013}), and the cuspidality can be deduced from \cite[Theorem 4.3]{Wallach84}. In this paper we give a more explicit function theoretic approach to Theorem \ref{thm2-intro}. For example, in Theorem \ref{thm-quaternionicity-B_l,v} we show that $\theta(\xi_f, \phi_0)$ is a quaternionic modular form on $\U(2,n)$ by explicitly verifying that $D_{\ell}^{\pm}\theta(P(\cdot,\mu_{-t}), \phi_0)\equiv 0$ as $P(\cdot,\mu_{-t})$ ranges over the space of weight $2\ell+2-n$ holomorphic Poincar\'e series. As a byproduct of this analysis, we obtain an explicit family of functions $\{B_{\ell,T}\colon \U(2,n)\to \VV_{\ell}\colon T\in V, \langle T,T\rangle>0\}$ satisfying $D_{\ell}^{\pm}B_{\ell,T}\equiv 0$. The functions $B_{\ell,T}$ furnish us with  integral representations for the functions $\cW_T$ of Theorem \ref{thm1-intro} (see \eqref{intergral-rep-of-W_v}). Our final theorem uses this integral representation of $\cW_T$, to show that the theta lifting from $\U(1,1)$ to $\U(2,n)$ preserves algebraicity of Fourier coefficients in the following sense. 
\begin{thm}[Theorem~\ref{thm-algebraicity}]
\label{thm3-intro}
    Let notation be as in Theorem \ref{thm2-intro}. Assume $L/\QQ$ is an algebraic extension such that $b_f(t)\in L$ for all $t\in \QQ_{>0}$. Let $L(\mu_{\infty})/L$ denote the extension obtained  by adjoining all roots of unity to $L$. If $T\in V_0$ satisfies $\langle T,T\rangle>0$ then $a_{\theta(\xi_f,\phi_0)}(T)\in L(\mu_{\infty})$. 
\end{thm}
We end the introduction by giving a brief outline of the structure of our paper. In \S\ref{section-preliminaries} we define a $\QQ$-rational form of $\U(2,n)$, which is denoted $\bfG$, and review several pieces of structure theory pertaining to $\bfG$. The explicit forms of the Schmid differential operators are given in \S\ref{The-Fourier-Expansion-of-QMFS-on-G}, as is the proof of Theorem \ref{thm1-intro}. After recalling some generalities regarding the theory of the theta correspondence in \S\ref{section-generalities-theta-correspondence}, we dedicate \S\ref{section-theta-lifting} to studying the quaternionic modular forms on $\bfG$ obtained as theta lifts from $\U(1,1)$. In particular, \S \ref{section-theta-lifting} contains a conditional proof of Theorem \ref{thm2-intro}. In \S\ref{section-Fourier-coefficients-of-theta-lift} we complete the proof of Theorem \ref{thm2-intro} and establish the algebraicity statement Theorem \ref{thm3-intro}.

\subsection*{Acknowledgements}
This work was initiated as part of a project group during the Arizona Winter School 2022. The authors thank The Southwest Center for Arithmetic Geometry, as this work would not have been possible if not for their stimulating winter school. The authors thank Aaron Pollack and Spencer Leslie for their guidance and support throughout the course of this work. Thanks are due to Jon Aycock, Poornima Belvotagi, Eric Chen, Chung-Hang Kwan, Nhat Hoang Le, Alice Lin, Angus McAndrew, Gyujin Oh, Alexander Schlesinger,  Shenrong Wang, and Katharine Woo for helpful conversations during the initial stages of this work. The authors thank the referee for the careful reading and helpful suggestions. 
PY is partially supported by an AMS-Simons Travel Grant.

\section{Preliminaries on the Group $\bfG$}
\label{section-preliminaries}

\subsection{Hermitian Spaces}
Fix an imaginary quadratic extension $E/\mathbb{Q}$. We regard $E$ as a subfield of $\CC$ via a fixed embedding $E\hookrightarrow \CC$. Write $\AA_E$ to denote the adele ring of $E$. The adele ring of $\QQ$ is denoted by $\AA=\AA_\QQ$. Let $x\mapsto \overline{x}$ be the non-trivial element of $\Gal(E/\QQ)$ and write $\tr_{E/\QQ}(x)=x+\overline{x}$ for the trace map from $E$ to $\QQ$. Let $\psi_{\QQ}=\psi=\prod_{v\leq \infty}\psi_v$ denote the standard additive character of $\QQ\backslash \AA$. So for $p<\infty$ , $\psi_p$ is the additive character of $\QQ_p$ of conductor $\ZZ_p$ and $\psi_{\infty}(x)=e^{2\pi i x}$.  Define $\psi_E\colon E\backslash \AA_E\to \CC$ by $\psi_E(x)=\psi_\QQ(\frac{1}{2}\tr_{E/\mathbb{Q}}(x))$.

Fix $n\in \ZZ_{\geq 1}$, and let $\mathbf{V}$ be a non-degenerate hermitian space over $E$ of signature $(2,n)$. Write $\langle \cdot ,\cdot \rangle$ for the hermitian form on $\mathbf{V}$ with the convention that $\langle\cdot, \cdot \rangle$ is conjugate linear in the second variable. The Hasse principle implies that $\bfV$ contains an isotropic line $\bfU$. Let $\bfU^{\vee}\subseteq \bfV$ denote a second isotropic line such that $\langle \bfU, \bfU^{\vee}\rangle\neq \{0\}$. Fix $b_1\in \bfU$ and $b_2\in \bfU^{\vee}$ satisfying $\langle b_1,b_2\rangle=1$ and consider the elements in $V:=\bfV\otimes_{\QQ}\RR$ given by $u_1=\frac{1}{\sqrt{2}}(b_1+b_2)$ and $v_n=\frac{1}{\sqrt{2}}(b_1-b_2)$. Define $\bfV_0=(\bfU\oplus \bfU^{\vee})^{\perp}$ so that $\bfV_0$ is a hermitian space of signature $(1,n-1)$ and
\begin{equation}
\label{Heisenberg-decomp-of-V}
\bfV=\bfU\oplus \bfV_0\oplus \bfU^{\vee}
\end{equation} 
We extend $\{u_1,v_n\}$ to a basis $\{u_1,u_2, v_1, \ldots,v_n\}$ of $V$ satisfying:
\begin{compactenum}[(i)]
\item[(i)] If $i,j\in \{1,2\}$ then $\langle u_i,u_j\rangle =\delta_{ij}$.
\item[(ii)] If $i,j\in \{1,\ldots, n\}$ then $\langle v_i,v_j\rangle =-\delta_{ij}$ and $\langle u_2,v_i\rangle =0$.
\item[(iii)] The subspace $V_0=\bfV_0\otimes_{\QQ}\RR$ is spanned by $\{u_2, v_1, \ldots , v_{n-1}\}$. 
\end{compactenum}
Finally, let $V_2^+=\CC\linspan\{u_1,u_2\}$ and $V_n^-=\CC\linspan\{v_1, \ldots, v_n\}$ so that
\begin{equation}
    \label{max-compact-identified}
    V=V_2^+\oplus V_n^-
\end{equation}
is a decomposition of $V$ into definite subspaces.  
\subsection{The Groups $\bfG$ and $\bfP$} Write $\mathbf{G}:=\U(\bfV)$ to denote the unitary group attached to $\bfV$ with the convention that $\bfG$ acts on the right of $\bfV$. More precisely, $\bfG$ is the algebraic group over $\QQ$ whose points on a test $\QQ$-algebra $R$ satisfy
$$
\mathbf{G}(R)=\{g\in \GL(\mathbf{V}\otimes_{\QQ}R) \colon \hbox{if $v,w\in \mathbf{V}\otimes_{\QQ}R$ then $\langle v\cdot g, w\cdot g\rangle =\langle v,w\rangle$}\}. 
$$
The Heisenberg parabolic $\bfP\leq \bfG$ is defined as the stabilizer in $\bfG$ of $\bfU$. A Levi factor $\bfM\leq \bfP$ may be defined as the stabilizer of $\bfU^{\vee}$ in $\bfP$. Through its action on the decomposition (\ref{Heisenberg-decomp-of-V}), $\bfM$ is identified as
$$
\bfM=\U(\bfV_0)\times \Res_{E/\QQ}\GG_{m}.
$$
Given a $\QQ$-algebra $R$ we write elements in $\bfM(R)$ as pairs $(h,z)$ with $h\in \U(\bfV_0)(R)$ and $z\in (R\otimes_{\QQ}E)^{\times}$. We normalize the coordinate $z$ so that the element $(h,z)\in \bfM(\QQ)$ acts on a vector $(u,v,u^{\vee})\in \bfU\oplus \bfV_0\oplus \bfU^{\vee}$ via the formula
$$
(u,v,u^{\vee})\cdot (h,z)=(z^{-1}u,v\cdot h, \overline{z}u^{\vee}).
$$
\subsection{The Lie algebras $\fg_0$ and $\fn_0$}
\label{The-Lie-algebras}
We  write $\overline{V}$ to denote the vector space obtained from $V$ by twisting the complex structure via the conjugation involution, i.e., $\overline{V}$ is the $\CC$-vector space spanned by the set $\{\overline{v}\colon v\in V\}$ subject to the relations $\overline{v+w}= \overline{v}+\overline{w}$ for all $v,w\in V$ and $\overline{\alpha}\cdot \overline{v}=\overline{\alpha\cdot v}$ for all $\alpha\in \CC$ and $v\in V$. Then $\overline{V}$ is identified with the $\CC$-linear dual $V^{\vee}$ of $V$ via the map $\overline{v}\mapsto \langle \cdot , v\rangle$. Write $\dag\colon V\otimes_{\CC}\overline{V}\to V\otimes_{\CC}\overline{V}$ for the $\CC$-semilinear involution induced by $(v\otimes \overline{w})^{\dag}=w\otimes \overline{v}$. Using the identification $\End(V)\simeq V\otimes_{\CC}V^{\vee}\simeq V\otimes_{\CC}\overline{V}$, $\fg_0:=\Lie(G)$ can be realized as 
\begin{equation}
\label{identification-g0}
\fg_0=(V\otimes_{\CC}\overline{V})^{\dag=-1}.
\end{equation}
Thus $\fg_0=\RR\linspan\{v\otimes \overline{w}-w\otimes \overline{v} \colon v,w\in V\}$. The primary benefit to this presentation of $\fg_0$ is that the right adjoint action of $G$ on $\fg_0$ is given by the simple formula
$$
(v\otimes \overline{w}-w\otimes \overline{v})\cdot g=vg\otimes \overline{wg}-wg\otimes \overline{vg}. 
$$ 

Let $\bfN\trianglelefteq \bfP$ denote the unipotent radical of $\bfP$ and set $N=\bfN(\RR)$. The commutator subgroup of $N$ is equal to the center $Z$ of $N$. Thus characters of $N$ are in bijection with characters of $N^{\mathrm{ab}}=N/Z$. In terms of the identification (\ref{identification-g0}), $Z=\{\exp(tib_1\otimes \overline{b_1}) \colon t\in \RR\}$. Let $\fn_0^{\mathrm{ab}}=\Lie(N^{\mathrm{ab}})$ so that
\begin{equation}
\label{n_0-abelianised}
\fn^{\mathrm{ab}}_0= \RR\linspan\{v\otimes \overline{b}_1-b_1\otimes \overline{v} \colon v\in V_0\}.
\end{equation}
The presentation (\ref{n_0-abelianised}) allows us to parameterize the set of unitary characters of $N$ using the vector space $V_0:=\bfV_0\otimes_{\QQ}\RR$ as follows. Given $T\in V_0$, let $\chi_{T,\infty}$ be the unique character of $N$ such that if $w\in V_0$ then
\begin{equation}
\label{defn-chiv0}
\chi_{T,\infty}(\exp(w\otimes \overline{b}_1-b_1\otimes \overline{w}))=e^{-2\pi i\Im(\langle T, w\rangle)}.
\end{equation} 
Since $(v,w)\mapsto -\Im(\langle v,w\rangle)$ defines a non-degenerate symplectic form on $V_0$, we obtain an identification 
\begin{equation}
\label{characters-of-N(R)}
    V_0\xrightarrow{\sim} \Hom(N,\CC^1), \qquad T\mapsto \chi_{T,\infty}.
\end{equation}
If $T\in \bfV_0$ we let $w\in \bfV_0(\AA)$ and define $\chi_{T}\colon \mathbf{N}(\QQ)\backslash \mathbf{N}(\AA)\to \CC^1$ via the formula $\chi_{T}(\exp(w\otimes \overline{b}_1-b_1\otimes \overline{w}))=\psi(-\Im(\langle T,w\rangle))$. The adelic analogue of (\ref{characters-of-N(R)}) is given by 
\begin{equation}
\label{characters-of-N(A)}
\bfV_0\xrightarrow{\sim} \Hom(\bfN(\QQ)\backslash \bfN(\AA), \CC^1), \qquad T\mapsto \chi_{T}. 
\end{equation}
Note that if $m=(h,z)\in \bfM(\QQ)$ and $T\in\bfV_0$, then
\begin{equation}
\label{M-character-action}
\chi_{T}(m\exp(w\otimes \overline{b}_1-b_1\otimes \overline{w})m^{-1})
=
\chi_{zT\cdot h}(\exp(w\otimes \overline{b}_1-b_1\otimes \overline{w})). 
\end{equation}
\subsection{The Cartan Decomposition in $\fg_0$}
\label{The-Cartan-Decomposition-in-g}
Define a maximal compact subgroup $K_{\infty}= \U(V_2^+)\times \U(V_n^-)$ as the stabilizer in $G$ of the decomposition (\ref{max-compact-identified}).
Let $\iota\in K_{\infty}$ be the element which acts as the identity on $V_2^+$ and acts as $-1$ on $V_n^-$. The Cartan involution associated to $K_{\infty}$ is given by $\theta:=\Ad(\iota)$. Therefore $\fk_0:=\Lie(K_{\infty})$ is the $+1$ eigenspace of $\theta$ and $\fp_0:=\fg_0^{\theta=-1}$ is a $K_{\infty}$-stable subspace of $\fg_0$. We wish to identify $\fp:=\fp_0\otimes_{\RR} \CC$ as a $K_{\infty}$-representation. As $\fp_0$ is itself a $\RR$-subspace of the $\CC$-vector space $V\otimes_{\CC}\overline{V}$, we write $\sqrt{-1}$ to denote the imaginary unit in the copy of $\CC$ appearing in the tensor $\fp_0\otimes_{\RR} \CC$.
 Given $j\in\{1,2\}$ and $k\in \{1, \ldots , n\}$, let
\begin{equation}
\label{eq-mathfrakp-basis}
[u_j\otimes \overline{v}_k]^{\pm}=(u_j\otimes \overline{v}_k-v_k\otimes \overline{u}_j)\otimes \left(\frac{1}{2}\right)+i(u_j\otimes \overline{v}_k+v_k\otimes \overline{u}_j)\otimes \left(\frac{\mp \sqrt{-1}}{2}\right).
\end{equation}
Define
$$
\fp^{\pm}=\CC\linspan\{[u_j\otimes \overline{v}_k]^{\pm} \colon j=1,2 \quad \hbox{and}\quad k=1, \ldots, n\}. 
$$
Then as a $K_{\infty}$-module,
\begin{equation}
\label{decomp-of-fracp}
\fp=\fp^+\oplus \fp^-.
\end{equation}
Moreover, we have $K_{\infty}$-module identifications
\begin{equation}
\label{identip}
\fp^+\simeq V_2^+\otimes \overline{V^-_n} \qquad\hbox{and} \qquad \fp^-\simeq \overline{V^+_2}\otimes V_n^-
\end{equation}
satisfying $[u_j\otimes \overline{v}_k]^+\mapsto u_j\otimes \overline{v}_k$ and $[u_j\otimes \overline{v}_k]^-\mapsto \overline{u}_j\otimes v_k$ respectively.

\section{The Fourier Expansion of Quaternionic Modular Forms on $\mathbf{G}$}
\label{The-Fourier-Expansion-of-QMFS-on-G}
Recall that $\mathbf{G}$ is the unitary group associated to a hermitian space over $E$ of signature $(2,n)$ and $\mathbf{P}=\mathbf{M} \mathbf{N}$ is the Heisenberg parabolic in $\mathbf{G}$ defined as the stabilizer of the isotropic line $\mathbf{U}=E b_1$. Fix a right invariant measure $dn$ on $\bfN(\QQ)\backslash \bfN(\AA)$. 
\subsection{Quaternionic Modular Forms on $\mathbf{G}$}
\label{subsection-quaternionic-modular-forms-on-G}
\indent Recall that $V=\mathbf{V}\otimes_{\QQ}\RR$, $G=\mathbf{G}(\RR)$, and $K_{\infty}= \U(V_2^+)\times \U(V_n^-)$ is a maximal compact subgroup of $G$ defined as the stabilizer of the orthogonal decomposition $V=V_2^+\oplus V_n^-$ (see (\ref{max-compact-identified})). Given $\ell\in \ZZ_{\geq 1}$, consider the $K_{\infty}$-representation 
$$
\VV_{\ell}:=\left(\Sym^{2\ell}V_2^+\otimes \det {}_{\U(V_2^+)}^{-\ell}\right)\boxtimes \mathbf{1}.
$$ 
Here $\mathbf{1}$ denotes the trivial representation of $\U(V_n^-)$. 
When $\ell\geq n$, Gross and Wallach \cite{GW96} construct an irreducible unitary representation $\Pi_{\ell}$ of $G$ such that $\Pi_{\ell}$ is discrete series and contains $\VV_{\ell}$ as its minimal $K$-type with multiplicity $1$. Similarly, if $1\leq \lfloor \frac{n-1}{2}\rfloor\leq \ell<n$ then Gross and Wallach (loc. cite) construct an irreducible unitary representation $\Pi_{\ell}$ which is no longer discrete series, but contains $\VV_{\ell}$ as its minimal $K$-type with multiplicity $1$. In \S \ref{The-Fourier-Expansion-of-QMFS-on-G} of this paper, $\ell$ will denote an integer satisfying either $\ell\geq n$ or $1\leq \lfloor \frac{n-1}{2}\rfloor\leq \ell<n$. \\
\indent The Schmid operators studied in this paper fall into a broad family of differential operators defined in \cite{Schmid89}. For the special case of the Schmid operators $\cD_{\ell}^+$ and $\cD_{\ell}^-$ associated to the representation $\Pi_{\ell}$, the definition proceeds as follows. In the notation of (\ref{decomp-of-fracp}) fix a basis $\{X_{\gamma}^+\}$ of $\fp^+$. The Killing form on $\fg$ induces a duality isomorphism $\fp^{-}\simeq (\fp^+)^{\vee}$. Let $\{X_{\gamma}^-\}$ denote the basis of $\fp^-$ dual to $\{X_{\gamma}^+\}$. Given $X\in \fg_0$ and $\varphi\in \cC^{\infty}(G, \VV_{\ell})$, the right regular action of $X$ on $\varphi$ is defined by $X\varphi(g)=\frac{d}{dt}(\varphi(g\exp(tX))\rvert_{t=0}$. This action extends linearly to $\fg$. \\
\indent Define operators $\tilde{\cD}_{\ell}^+$ and $\tilde{\cD}_{\ell}^-$ via
$$
\tilde{\cD}_{\ell}^{\pm}\colon \cC^{\infty}(G, \VV_{\ell})\to \cC^{\infty}(G, \VV_{\ell}\otimes_{\CC}\fp^{\mp}), \quad \varphi\mapsto \tilde{\cD}_{\ell}^{\pm}\varphi:=\sum_{\gamma}X_{\gamma}^{\pm}\varphi\otimes X_{\gamma}^{\mp}.
$$
For $v\in \{-\ell, \ldots, \ell\}$, let
\begin{equation}
\label{definition-basis-V-ELL}
[u_1^{\ell-v}]=\frac{u_1^{\ell-v}}{(\ell-v)!} \qquad \hbox{and} \qquad [u_2^{\ell+v}]=\frac{u_2^{\ell+v}}{(\ell+v)!}
\end{equation}
so that $\cB=\{[u_1^{\ell-v}][u_2^{\ell+v}]\colon v=-\ell,\ldots, 0,\ldots, \ell\}$ is a basis of $\mathbb{V}_{\ell}$. In terms of the basis $\cB$, a $\U(V_2^+)$-equivariant contraction $\Sym^{2\ell}V_2^+\times \overline{V_2^+}\to \Sym^{2\ell-1}V_2^+$is given by
$$
[u_1^{\ell-v}][u_2^{\ell+v}]\otimes \overline{u_1}\mapsto [u_1^{\ell-v-1}][u_2^{\ell+v}] \quad \hbox{and} \quad [u_1^{\ell-v}][u_2^{\ell+v}]\otimes \overline{u_2}\mapsto [u_1^{\ell-v}][u_2^{\ell+v-1}].
$$
We have another contraction $\Sym^{2\ell}V_2^+\times V_2^+\to \det_{\U(V_2^+)}^{-1}\otimes \Sym^{2\ell-1}V_2^+$ defined by
$$
[u_1^{\ell-v}][u_2^{\ell+v}]\otimes u_1\mapsto -[u_1^{\ell-v}][u_2^{\ell+v-1}] \quad \hbox{and} \quad [u_1^{\ell-v}][u_2^{\ell+v}]\otimes u_2\mapsto [u_1^{\ell-v-1}][u_2^{\ell+v}].
$$
These contractions yield $K_{\infty}$-equivariant projections
$$
\begin{cases}
\pi^+\colon \VV_{\ell}\otimes \fp^-=(\Sym^{2\ell}V_2^+\otimes \det{}_{\U(V_2^+)}^{-\ell}\otimes \overline{V_2^+})\boxtimes V_n^-\to (\Sym^{2\ell-1}V_2^+\otimes \det{}_{\U(V_2^+)}^{-\ell})\boxtimes V_n^-,
\\
\pi^-\colon \VV_{\ell}\otimes \fp^+=(\Sym^{2\ell}V_2^+\otimes \det{}_{\U(V_2^+)}^{-\ell}\otimes V_2^+)\boxtimes \overline{V_n^-}\to (\Sym^{2\ell-1}V_2^+\otimes \det{}_{\U(V_2^+)}^{-(\ell+1)})\boxtimes \overline{V_n^-}
\end{cases}
$$
The Schmid operators $\cD_\ell^+$ and $\cD_\ell^-$ are defined as
\begin{equation}
\label{defn-Schmid-differential}
\begin{cases}
\cD_{\ell}^+\colon \cC^{\infty}(G, \VV_{\ell})\to \cC^{\infty}(G,(\Sym^{2\ell-1}V_2^+\otimes \det{}_{\U(V_2^+)}^{-\ell})\boxtimes V_n^- ), \quad \varphi\mapsto \pi^+\circ \tilde{\cD}_{\ell}^+\varphi, \\
\cD_{\ell}^-\colon \cC^{\infty}(G, \VV_{\ell})\to  \cC^{\infty}(G,(\Sym^{2\ell-1}V_2^+\otimes \det{}_{\U(V_2^+)}^{-(\ell+1)})\boxtimes \overline{V_n^-}), \quad \varphi\mapsto \pi^-\circ \tilde{\cD}_{\ell}^-\varphi. 
\end{cases}
\end{equation}
With the operators $\cD_{\ell}^{\pm}$ in hand, we may define quaternionic modular forms on $\mathbf{G}$. 
\begin{definition}
\label{definition-quaternionic-modular-form}
A \textit{weight $\ell$ quaternionic modular form on $\mathbf{G}$} is a smooth function $F\colon \mathbf{G}(\AA) \to \VV_{\ell}$ of moderate growth such that:
\begin{compactenum}[(i)] 
\item If $\gamma\in \mathbf{G}(\QQ)$ and $g\in \mathbf{G}(\AA)$ then $F(\gamma g)=F(g)$.
\item If $k\in K_{\infty}$ and $g\in \mathbf{G}(\AA)$ then $F(gk)=F(g)\cdot k$. 
\item The functions $\cD_{\ell}^+F\rvert_G$ and $\cD_{\ell}^-F\rvert_G$ vanish identically on $G$. 
\end{compactenum}
\end{definition}
\subsection{Generalized Whittaker Coefficients}
\label{subsection-Generalized-Whittaker-Coefficients}
We now introduce the class of generalized Whittaker functions which feature in our study of the Fourier expansion of quaternionic modular forms on $\bfG$. 
\begin{definition}
\label{definition-generalized-whittaker-coefficient}
Let $N=\bfN(\RR)$ and fix $T\in V_0$. Write 
$$
\cC_{N,T}^{\mathrm{md}}(G, \VV_{\ell})_{K_{\infty}}^{\cD_{\ell}=0}
$$ to denote the space of smooth moderate growth functions $\cW_{\chi}\colon G\to \VV_{\ell}$ satisfying:
\begin{compactenum}[(a)] 
\item If $k\in K_{\infty}$ and $g\in G$ then $\cW_{\chi}(gk)=\cW_{\chi}(g)\cdot k$. 
\item If $n\in N$ and $g\in G$ then $\cW_{\chi}(ng)=\chi_{T,\infty}(n)\cW_{\chi}(g)$. 
\item The functions $\cD_{\ell}^+\cW_{\chi}$ and $\cD_{\ell}^-\cW_{\chi}$ vanish identically on $G$. 
\end{compactenum}
\end{definition}
\begin{thm}
\label{Theorem-Multiplicity-At-Most-1}
Suppose $T\in V_0$ is non-zero. Define a function 
$$
\beta_T\colon M\to \CC,\qquad \beta_T(h,z)=\frac{4\pi}{\sqrt{2}}|\langle u_2, zT\cdot h\rangle|.
$$
We have
$$
\dim_{\CC}\left(\cC_{N,T}^{\mathrm{md}}(G, \VV_{\ell})_{K_{\infty}}^{\cD_{\ell}=0}\right)
=
\begin{cases} 1, &\hbox{if $\langle T,T\rangle \geq 0$}, \\
0, &\hbox{if $\langle T,T\rangle< 0$.}
\end{cases}
$$
If $\langle T,T\rangle \geq 0$, then there exists a unique function $\cW_T\in \cC_{N,T}^{\mathrm{md}}(G, \VV_{\ell})_{K_{\infty}}^{\cD_{\ell}=0}$ satisfying
\begin{equation}
\label{Main-Formula-Mult-One}
\small \cW_{T}(h,z):=\sum_{-\ell\leq v\leq \ell}|z|^{2\ell+2}\left(\frac{|\beta_T(h,z)|}{\beta_T(h,z)}\right)^vK_v\left(|\beta_T(h,z)|\right)[u_1^{\ell-v}][u_2^{\ell+v}]
\end{equation}
for all $(h,z)\in M$.
Here $K_v$ denotes the $K$-Bessel function $K_v(x)=\frac{1}{2}\int_0^{\infty}t^{v-1}e^{-x(t+t^{-1})}dt$.
\end{thm}

\begin{remark}
    The $n=1$ case of the above theorem is established in \cite{KT95}. 
\end{remark}
\begin{remark}
\label{rmk2-qmfs}
Fix $T\in V_0$ and let $\Pi_{\ell}^{\infty}$ be the space of smooth vectors in $\Pi_{\ell}$ taken relative to its Fr\'echet topology. The space of \textit{generalized Whittaker functionals} is 
$$
\mathrm{Wh}_N(\Pi_{\ell},\chi_{T,\infty}):=\Hom_N^{\mathrm{cont}}(\Pi_{\ell}^{\infty},\chi_{T,\infty}).
$$ 
In \cite[Theorem 16]{Wallach03}, Wallach calculates the dimension of $\mathrm{Wh}_N(\Pi_{\ell},\chi_{T,\infty})$ in the case when $\Pi_{\ell}$ is a discrete series representation and the character $\chi_T$ is generic. The result of (loc. cite) applies to a rather general class of real Lie groups, however, in the case of the group $\U(2,n)$, we may summarize the result of (loc. cite) as the statement that if  $\ell\geq n$, then
\begin{equation}
\label{Wallach-Theorem-16}
\dim_{\CC} \mathrm{Wh}_N(\Pi_{\ell},\chi_{T,\infty})=
\begin{cases}
    1, &\hbox{if $\langle T,T\rangle > 0$,} \\
    0, &\hbox{if $\langle T,T\rangle <0$.}
\end{cases}
\end{equation} 
Here, we have applied the fact that the character $\chi_{T,\infty}$ is admissible in the sense of \cite{Wallach03} if and only if $\langle T,T\rangle >0$. To see this equivalence, let $N_1$ denote the subgroup of $N$ with Lie algebra spanned by $\{b_1\otimes \overline{u}_2-u_2\otimes \overline{b}_1, i(b_1\otimes \overline{u}_2+u_2\otimes \overline{b}_1), ib_1\otimes \overline{b}_1\}$. If $\langle T,T\rangle >0$ and $m\in \mathbf{M}^{\mathrm{der}}(\RR)$ then $\langle T\cdot m, u_2\rangle \neq 0$. Hence, the character $\chi_{T\cdot m,\infty}$ is non-trivial on $N_1$, and $\chi_T$ is admissible in the sense of (loc. cite). \\
\indent Conversely, if $\langle T, T\rangle \neq 0$ and $\chi_{T\cdot m,\infty}$ is non-trivial on $N_1$ for all $m\in \mathbf{M}^{\mathrm{der}}(\RR)$, then $\langle T\cdot m, u_2\rangle \neq 0$ for all $m\in \mathbf{M}^{\mathrm{der}}(\RR)$. Applying the Witt extension theorem, we conclude that $\langle T, T\rangle >0$, which finishes the proof that $\langle T, T\rangle>0$ if and only if $\chi_{T,\infty}$ is admissible. \\
\indent Now, if $\{v_i\}_{-\ell\leq v\leq \ell}$ is a basis for $\VV_{\ell}$ and $\{v_i^{\vee}\}_{-\ell\leq v\leq \ell}$ is the dual basis of $\VV_{\ell}^{\vee}\simeq \VV_{\ell}$ then the map 
\begin{equation}
\label{injective-map}
\mathrm{Wh}_N(\Pi_{\ell},\chi_{T,\infty})\to \mathcal{C}_{N,T}^{\mathrm{md}}(G,\VV_{\ell})_{K_{\infty}}^{\cD_{\ell}=0}, \quad \cL\mapsto \left(g\mapsto \sum_{i=-\ell}^{\ell}\cL(g\cdot v_i) v_i^{\vee}\right)
\end{equation}
is clearly injective. Our proof of Theorem \ref{Theorem-Multiplicity-At-Most-1} does not depend upon \cite[Theorem 16]{Wallach03}, however, in light of \eqref{injective-map} and \eqref{Wallach-Theorem-16}, the reader may observe that these results are closely related. 
\end{remark}
We dedicate \S \ref{subsection-Explicating-the-Schmid-Equations} and \S \ref{subsection-solving-the-schmid-equations} to a detailed presentation of the proof of Theorem \ref{Theorem-Multiplicity-At-Most-1}. To give an indication of the method, suppose $T\in V_0$ satisfies $T\neq 0$ and let $\mathcal{W}_T\colon G\to \VV_{\ell}$ be a smooth function of moderate growth satisfying hypotheses (a) and (b) of Definition \ref{definition-generalized-whittaker-coefficient}. We write $\{W_{T,v}\colon G\to \VV_{\ell}\}_{-\ell\leq v\leq \ell}$ for the unique family of scalar valued functions such that if $g\in G$ then
\begin{equation}\label{defn-WTv}
\cW_T(g)=\sum_{-\ell\leq v \leq \ell}W_{T,v}(g)[u_1^{\ell-v}][u_2^{\ell+v}].
\end{equation}
The first step in the proof of Theorem \ref{Theorem-Multiplicity-At-Most-1} is to consider the restrictions $\mathcal{D}_{\ell}^{+}\cW_T\rvert_M$ and $\mathcal{D}_{\ell}^{-}\cW_T\rvert_M$. We express the conditions $\mathcal{D}_{\ell}^{\pm}\mathcal{W}_T\rvert_M\equiv 0$ as a system of differential equations involving the functions $\mathcal{W}_{T,v}$ (see Proposition \ref{System-Scalar-DEs}). In \S \ref{subsubsection-establishing-the-candidate-solution} we arrive at the formula (\ref{Main-Formula-Mult-One}) by solving a subset of the system of equations in Proposition \ref{System-Scalar-DEs}. Formula (\ref{Main-Formula-Mult-One}) is not smooth as a function on $M$ unless $\langle T,T\rangle \geq 0$, (see Proposition \ref{Proposition-a-large-portion}) and this observation, together with our analysis in \S \ref{subsubsection-establishing-the-candidate-solution} implies the $(T,T)<0$ case of Theorem \ref{Theorem-Multiplicity-At-Most-1} as well as the inequality $\dim \cC_{N,{T}}^{\mathrm{md}}(G, \VV_{\ell})_{K_{\infty}}^{\cD_{\ell}=0}\leq 1$ (see Proposition \ref{Proposition-a-large-portion}). It remains to show that if $\langle T,T\rangle \geq 0$ then (\ref{Main-Formula-Mult-One}) defines an element of $\cC_{N,T}^{\mathrm{md}}(G, \VV_{\ell})_{K_{\infty}}^{\cD_{\ell}=0}$.  This final check is carried out in \S \ref{subsubsection-verifying-schmid-equations}. 
\begin{cor}\label{cor1} Suppose $F\colon \bfG(\AA)\to \VV_{\ell}$ is a weight $\ell$ quaternionic modular form. Let $\bfZ$ be the center of the Heisenberg unipotent radical $\bfN$. Write $F_{\bfZ}$ (resp. $F_{\bfN}$) for the constant term of $F$ along $\bfZ$ (resp. $\bfN$). 
There exist locally constant functions $
\{a_{T}(F, \cdot )\colon \bfG(\AA_\fin)\to \CC\}_{T\in \bfV_0\colon \langle T,T\rangle\geq 0}
$ such that if ${g_\fin}\in \bfG(\AA_\fin)$ and $g_{\infty}\in \bfG(\RR)$ then 
\begin{equation}
\label{FE-Noncuspidal-QMF1}
F_{\bfZ}(g_\fin g_{\infty})=F_{\bfN}(g_\fin g_{\infty})+\sum_{T\in \bfV_0-\{0\}\colon \langle T,T\rangle \geq 0}a_T(F,g_\fin)\cW_T(g_{\infty}).
\end{equation}
Moreover, if $F$ is cuspidal then the Fourier expansion of $F_{\bfZ}$ takes the form
\begin{equation}
\label{FE-Cuspidal-QMF}
F_{\bfZ}({g_\fin}g_{\infty})=\sum_{T\in \bfV_0\colon \langle T, T\rangle> 0}a_{T}(F,{g_\fin})\cW_{T}(g_{\infty}).
\end{equation}
\end{cor}  
\begin{proof} The proof of formula (\ref{FE-Noncuspidal-QMF1}) is standard given Theorem \ref{Theorem-Multiplicity-At-Most-1}. Indeed, the domain of integration $\bfN(\QQ)\backslash \bfN(\AA)$ is compact. So if $g_f\in \bfG(\AA_{\mathrm{fin}})$, then the function
$$
F_T(\cdot\,g_f)\colon \bfG(\RR)\to \mathbb{V}_{\ell}, \qquad F_T(g_{\infty}g_f)= \int_{\bfN(\QQ)\backslash \bfN(\AA)}F(ng_{\infty}g_{\mathrm{fin}})\chi_{T}(n)^{-1}\,dn
$$
satisfies the hypotheses of Definition \ref{definition-generalized-whittaker-coefficient}. Applying Theorem \ref{Theorem-Multiplicity-At-Most-1}, it follows that $F_T\equiv 0$ whenever $\langle T, T\rangle <0$. Similarly, Theorem \ref{Theorem-Multiplicity-At-Most-1} implies that there exists a scalar $a_T(F, g_\fin)\in \CC$  such that $F_T(g_fg_{\infty})=a_{T}(F,{g_\fin})\cW_{T}(g_{\infty})$ whenever $T\neq 0$ and $\langle T, T\rangle \geq 0$. The proof of \eqref{FE-Noncuspidal-QMF1} now follows by Fourier expanding $F_{\mathbf{Z}}$ in characters of $\mathbf{N}/\mathbf{Z}$. For the proof of (\ref{FE-Cuspidal-QMF}), assume $F$ is cuspidal. Let $T\in \bfV_0$ be non-zero and suppose $\langle T, T\rangle=0$. Fixing ${g_\fin}\in \bfG(\AA_\fin)$ it suffices to show that $a_{T}(F, {g_\fin} )= 0$. We consider the equality 
\begin{equation}
\label{comparison-of-hands}
a_{T}(F,{g_\fin})\cW_{T}(g_{\infty}) =\int_{\bfN(\QQ)\backslash \bfN(\AA)}F(n{g_\fin}g_{\infty})\overline{\chi_{T}(n)}dn.
\end{equation}
On the one hand, $F$ is a cusp form and thus bounded on $\bfG(\AA)$. It follows that the left-hand side of (\ref{comparison-of-hands}) is bounded as a function of $g_{\infty}$ relative to the $K_{\infty}$ invariant norm on $\mathbb{V}_{\ell}$. On the other hand, if $-\ell\leq v\leq \ell$ then the $K$-Bessel function $K_v(x)$ is unbounded as $x\to 0^+$. Denote $M^{z=1}=\{(h,z)\in M: z=1\}$. Since $M^{z=1}\to \RR_{>0}$, $(h, 1)\mapsto |\langle u_2\cdot h^{-1},T\rangle|$ is not bounded away from zero, \eqref{Main-Formula-Mult-One} implies that $\mathcal{W}_T(g_{\infty})$ is unbounded relative to the $K_{\infty}$ invariant norm on $\mathbb{V}_{\ell}$. Hence $a_{T}(F, {g_\fin} )\equiv 0$ and the proof is complete.
\end{proof}
\subsection{Explicating the Schmid Equations}
\label{subsection-Explicating-the-Schmid-Equations}
In this subsection we begin the proof of Theorem \ref{Theorem-Multiplicity-At-Most-1}. Throughout \S \ref{subsection-Explicating-the-Schmid-Equations} we fix $T\in V_0$ such that $T\neq 0$ and suppose $\cW_T\colon G\to \RR$ is a smooth function satisfying conditions (a) and (b) of Definition \ref{definition-generalized-whittaker-coefficient}. 
\subsubsection{Iwasawa Coordinates}
Let $\fn$ (resp. $\fm$) denote the complexified Lie algebra of $N$ (resp. $M$).
The next lemma will be used to study the restrictions $\cD_{\ell}^{\pm}\cW_T\rvert_{M}$. The proof is a direct computation. 
\begin{lemma}
\label{Iwasawa-Coordinates}
In terms of the Iwasawa decomposition $\fg=\fn+\fm+\fk$, the element $[u_j\otimes \overline{v_k}]^{\pm}$ is expressed as $[u_j\otimes\overline{v_k}]^{\pm}=X^{\pm}+Y^{\pm}+H^{\pm}$
where $X^{\pm}\in \fn$, $Y^{\pm}\in \fm$, and $H^{\pm}\in \fk$ and the triple $(X^{\pm},Y^{\pm},H^{\pm})$ is defined as follows:
\begin{itemize}
    \item[(a)] For $j=1$ and $k=n$, $X^{\pm}=(2ib_1\otimes \overline{b}_1)\otimes \left(\frac{\mp\sqrt{-1}}{2}\right)$, $Y^{\pm}=(b_2\otimes \overline{b}_1-b_1\otimes \overline{b}_2)\otimes \left(\frac{1}{2}\right)$, \\ and $H^{\pm}=i(b_1\otimes \overline{b}_1+b_2\otimes \overline{b}_2)\otimes \left(\frac{\pm \sqrt{-1}}{2}\right)$.
    \item[(b)] For $j=1$ and $1\leq k< n$, $X^{\pm}=\frac{2}{\sqrt{2}}[b_1\otimes \overline{v_k}]^{\pm}$, $Y^{\pm}=0$, and $H^{\pm}=-[v_n\otimes \overline{v_k}]^{\pm}$.
    \item[(c)] For $j=2$ and $k=n$, $X^{\pm}=\frac{2}{\sqrt{2}}[u_2 \otimes \overline{b_1}]^{\pm}$, $Y^{\pm}=0$, and $H^{\pm}=-[u_2\otimes \overline{u_1}]^{\pm}$.
    \item[(d)] For $j=2$ and $1\leq k <n$, $X^{\pm}=0$, $Y^{\pm}=[u_2\otimes \overline{v}_k]^{\pm}$, $H^{\pm}=0$. 
\end{itemize}
\end{lemma}
\subsubsection{The Right Regular Actions of $\fp^{\pm}$}
We apply Lemma \ref{Iwasawa-Coordinates} to calculate the right regular action of $\fp$ on $\cW_T$.
Throughout elements $m=(h,z)\in M$ are expressed as triples $(h,w,s)$ with $z=we^{is}$, $w\in \RR_{>0}$, and $s\in [0,2\pi)$. The function $\cW_{T,v}$ is defined in (\ref{defn-WTv}).
\begin{prop}\label{right-regular-action-of-p}\
\begin{compactenum}[(a)] 
\item If $m=(h,w,s)$ and $z=we^{is}$ then $[u_j\otimes \overline{v_k}]^+\cW_T(m)$ equals
$$
\begin{cases}
-\frac{1}{2}w\partial_w \cW_T(h,w,s)- \frac{1}{2}\sum_{-\ell\leq v\leq \ell}W_{T,v}(h,w,s)v[u_1^{\ell-v}][u_2^{\ell+v}], &\hbox{if $j=1$ and $k=n$,} \\
\frac{-2\pi\overline{\langle v_k, zT\cdot h\rangle}}{\sqrt{2}}\cW_T(h,z), &\hbox{if $j=1$ and $1\leq k <n$}, \\
\frac{-2\pi\langle u_2, zT\cdot h\rangle}{\sqrt{2}}\cW_T(m)+\sum_{-\ell\leq v\leq \ell}W_{T,v}(m)(\ell+ v+1)[u_1^{\ell-v- 1}][u_2^{\ell+v+ 1}], &\hbox{if $j=2$ and $k=n$,} \\ 
\sum_{-\ell\leq v\leq \ell}[u_2\otimes \overline{v_k}]^{+}W_{T,v}(m), &\hbox{if $j=2$ and $1\leq k <n$.}
\end{cases}
$$
\item If $m=(h,w,s)$ and $z=we^{is}$ then $[u_j\otimes \overline{v_k}]^-\cW_T(m)$ equals
$$
\begin{cases}
-\frac{1}{2}w\partial_w \cW_T(h,w,s)+ \frac{1}{2}\sum_{-\ell\leq v\leq \ell}W_{T,v}(h,w,s)v[u_1^{\ell-v}][u_2^{\ell+v}], &\hbox{if $j=1$ and $k=n$,} \\
\frac{2\pi \langle v_k, zT\cdot h\rangle}{\sqrt{2}}\cW_T(h,z), &\hbox{if $j=1$ and $1\leq k <n$}, \\
\frac{2\pi\overline{\langle u_2, zT\cdot h\rangle}}{\sqrt{2}}\cW_T(m)-\sum_{-\ell\leq v\leq \ell}W_{T,v}(m)(\ell- v+1)[u_1^{\ell-v+1}][u_2^{\ell+v- 1}], &\hbox{if $j=2$ and $k=n$,} \\ 
\sum_{-\ell\leq v\leq \ell}[u_2\otimes \overline{v_k}]^{-}W_{T,v}(m), &\hbox{if $j=2$ and $1\leq k <n$.}
\end{cases}
$$
\end{compactenum}
\end{prop}
\begin{proof} The proof is a direct computation. To give an example of the method of proof we prove the formulas for $[u_1\otimes\overline{v_n}]^{\pm}\cW_T$. \\
\indent Let $X^{\pm}$, $Y^{\pm}$ and $H^{\pm}$ be as in Lemma \ref{Iwasawa-Coordinates}(a). Then $X^{\pm}\in \Lie(Z)$ and since $\chi_T$ is trivial on $Z$, we may use property (b) of Definition \ref{definition-generalized-whittaker-coefficient} to deduce that 
\begin{equation}
\label{eqnactp1}
X^{\pm}\cW_T\equiv 0.
\end{equation}
If $m=(h,w,s)$ and $t\in \RR$ then $m\exp(t(b_2\otimes\overline{b_1}-b_1\otimes \overline{b_2})=(h,e^{-t}w,s)$. Therefore
\begin{equation}
\label{eqnactp2}
Y^{\pm}\cdot \cW_T(m)
=
-\frac{1}{2}\left(e^{it}w\frac{\partial}{\partial w} \cW_T(h,e^{-t}w,s)\right)\rvert_{t=0} =
-\frac{1}{2}w\frac{\partial}{\partial w}\cW_T(h,w,s).
\end{equation} 
Finally, if $t\in \RR$ then $[u_1^{\ell-v}][u_2^{\ell+v}]\cdot \exp(ti(b_1\otimes \overline{b_1}+b_2\otimes \overline{b_2}))=e^{\sqrt{-1}vt}[u_1^{\ell-v}][u_2^{\ell+v}]$. Thus by property (a) of Definition \ref{definition-generalized-whittaker-coefficient} we have
\begin{align}
\label{eqnactp3}
H^{\pm}\cdot \cW_T(m)
&= \frac{\pm\sqrt{-1}}{2}\frac{d}{dt}\left(\sum_{-\ell\leq v\leq \ell}e^{\sqrt{-1}vt}W_{T,v}(m)[u_1^{\ell-v}][u_2^{\ell+v}]\right)\rvert_{t=0} \nonumber \\
&=\frac{\mp 1}{2}\sum_{-\ell\leq v\leq \ell}W_{T,v}(m)v[u_1^{\ell-v}][u_2^{\ell+v}]
\end{align}
By Lemma \ref{Iwasawa-Coordinates} (a), $[u_1\otimes \overline{v_n}]^{\pm}\cW_T=X^{\pm}\cW_T+H^{\pm}\cW_T+Y^{\pm}\cW_T$. As such, the formulas for $[u_1\otimes \overline{v_n}]^{\pm}\cW_T$ may be derived from (\ref{eqnactp1}), (\ref{eqnactp2}), and (\ref{eqnactp3}).
\end{proof}
\subsubsection{Expansion of the Schmid operator}
\label{subsubsection-expansion-of-schmid}
We now apply the result of Proposition \ref{right-regular-action-of-p} to re-express the condition $\cD_{\ell}^{\pm}\cW_T\rvert_M\equiv 0$ as a system of scalar equations.
\begin{prop} Let $m=(h,w,s)\in M$ and set $z=we^{is}$.
\label{System-Scalar-DEs}
\begin{compactenum}[(a)]
\item We have $\cD_{\ell}^+\cW_T\rvert_M\equiv 0$ if and only if 
$$
\begin{cases}
(w\partial_w-2(\ell+1)-v)W_{T,v}(m)+\frac{4\pi}{\sqrt{2}}\langle u_2, zT\cdot h\rangle W_{T,v+1}(m)=0, \quad \hbox{if $v=-\ell,\ldots, \ell-1$,} \\
[u_2\otimes \overline{v}_k]^+W_{T,v}(m)-\frac{2\pi}{\sqrt{2}}\overline{\langle v_k, zT\cdot h\rangle}W_{T,v-1}(m)=0, \quad \hbox{if $1\leq k<n$ and $-\ell< v\leq \ell$.} 
\end{cases}$$
\item We have $\cD_{\ell}^-\cW_T\rvert_M\equiv 0$ if and only if 
$$
\begin{cases}
 (w\partial_w-2(\ell+1)+v)W_{T,v}(m)+\frac{4\pi}{\sqrt{2}}\overline{\langle u_2, zT\cdot h\rangle}W_{T,v-1}(m) =0, \quad \hbox{if $v=-\ell+1,\ldots, \ell$}, \\
 [u_2\otimes \overline{v}_k]^-W_{T,v}(m)-\frac{2\pi}{\sqrt{2}}\langle v_k, zT\cdot h\rangle W_{T,v+1}(m)=0, \quad \hbox{if $1\leq k<n$ and $-\ell\leq v< \ell$.}
\end{cases}
$$
\end{compactenum}
\end{prop}
\begin{proof} We prove part (a) of the proposition. Applying the identification $\fp^-\simeq \overline{V_2^+}\otimes V_n^-$ (\ref{identip}) and Proposition \ref{right-regular-action-of-p}(a) we obtain
\begin{align*}
\tilde{D}_{\ell}^{+}W_T(m)
&=  -\sum_{-\ell\leq v\leq \ell}\sum_{k=1}^{n-1}\frac{2\pi\overline{\langle v_k, zT\cdot h\rangle}}{\sqrt{2}}W_{T,v}(h,z)[u_1^{\ell-v}][u_2^{\ell+v}]\otimes \overline{u}_1\boxtimes v_k \\ &+ \sum_{-\ell\leq v\leq \ell}\sum_{k=1}^{n-1}[u_2\otimes \overline{v}_k]^+W_{T,v}(m)[u_1^{\ell-v}][u_2^{\ell+v}]\otimes \overline{u}_2\boxtimes v_k \\
&-\frac{1}{2}w\partial_w W_T(m)\otimes \overline{u}_1\boxtimes v_n- \frac{1}{2}\sum_{-\ell\leq v\leq \ell}W_{T,v}(m)v[u_1^{\ell-v}][u_2^{\ell+v}]\otimes \overline{u}_1\boxtimes v_n \\
&-\frac{2\pi\langle u_2, zT\cdot h\rangle}{\sqrt{2}}W_T(m)\otimes \overline{u}_2\boxtimes v_n \\
&+\sum_{-\ell\leq v\leq \ell}W_{T,v}(m)(\ell+ v+1)[u_1^{\ell-v- 1}][u_2^{\ell+v+ 1}]\otimes \overline{u}_2\boxtimes v_n.
\end{align*}
Applying the contraction operator $\pi^+$ defined in \S \ref{subsection-quaternionic-modular-forms-on-G} it follows that 
\begin{align*}
D^+_{\ell}W_T(m) 
&= 
 -\sum_{-\ell< v\leq \ell}\sum_{k=1}^{n-1}\frac{2\pi\overline{\langle v_k, zT\cdot h\rangle}}{\sqrt{2}}[u_1\otimes \overline{v}_k]^+W_{T,v-1}(m)[u_1^{\ell-v}][u_2^{\ell+v-1}]\boxtimes v_k \\
&+\sum_{-\ell< v\leq \ell}\sum_{k=1}^{n-1}[u_2\otimes \overline{v}_k]^+W_{T,v}(m)[u_1^{\ell-v}][u_2^{\ell+v-1}]\boxtimes v_k \\
&-\frac{1}{2}\sum_{-\ell < v \leq \ell}\left(w\partial_wW_{T,v-1}(m)+(v-1)W_{T,v-1}(m)\right)[u_1^{\ell-v}][u_2^{\ell+v-1}]\boxtimes v_n \\
&-\frac{2\pi}{\sqrt{2}}\sum_{-\ell< v\leq \ell}\langle u_2, zT\cdot h\rangle W_{T,v}(m)[u_1^{\ell-v}][u_2^{\ell+v-1}]\boxtimes v_n \\ 
&+\sum_{\ell<v\leq \ell}W_{T,v-1}(m)(\ell+v)[u_1^{\ell-v}][u_2^{\ell+v-1}]\boxtimes v_n. 
\end{align*}
One obtains the system of equations in (a) by equating the coefficients of the terms $[u_1^{\ell-v}][u_2^{\ell+v-1}]\boxtimes v_k$ to zero in the above expansion of $\cD_{\ell}^+\cW_T\rvert_M$. We omit the proof of the part (b) since it is verified similarly.
\end{proof}
\subsection{Solving the Schmid Equations}
\label{subsection-solving-the-schmid-equations}
Throughout this subsection $T\in V_0$  is non-zero and we suppose  $\cW_T\colon G\to \VV_{\ell}$ is a function of moderate growth satisfying conditions (a) and (b) of Definition \ref{definition-generalized-whittaker-coefficient} together with the hypothesis $\cD_{\ell}^{\pm}\cW_T\rvert_M\equiv 0$.
\subsubsection{Establishing the Candidate Solution}
\label{subsubsection-establishing-the-candidate-solution}
For $-\ell\leq v\leq \ell$ define $f_{T,v}\colon M\to \CC$ as

\begin{equation}
    \label{defn-ftv}
 f_{T,v}(h,w,s):=w^{-2(\ell+1)}W_{T,v}(h,w,s).
\end{equation}
Given $(h,z)\in M$ we define $\beta_T(h,z)=\frac{4\pi}{\sqrt{2}}\langle u_2,zT\cdot h\rangle$. Then Proposition \ref{System-Scalar-DEs} implies
\begin{equation}
\label{equations-for-ftv}
\begin{cases}
(w\partial_w-v)f_{T,v}(h,w,s)=-\beta_T(h,w,s)f_{T,v+1}(h,w,s), &-\ell\leq v< \ell \\ 
(w\partial_w+v)f_{T,v}(h,w,s)=-\overline{\beta_T(h,w,s)}f_{T,v-1}(h,w,s), &-\ell< v \leq \ell. 
\end{cases}
\end{equation}
Thus if $-\ell\leq v \leq v$ then 
\begin{equation}
\label{Bessel-Equation}
((w\partial_w)^2-v^2)f_{T,v}(h,w,s)=|\beta_T(h,w,s)|^2f_{T,v}(h,w,s).
\end{equation} 
The condition $\beta_T(m)\neq 0$ is non-empty and Zariski open on $M$. Fix $(h,1,s)\in M$ such that $\beta_T(h,1,s)\neq 0$. Under the substitution $u=w\cdot |\beta_T(h,1,s)|$, (\ref{Bessel-Equation}) reduces to the Bessel equation $\partial_u^2f_{T,v}+\frac{1}{u}\partial_uf_{T,v}-(1+\frac{v^2}{u^2})f_{T,v}=0$ \cite[8.494, pg. 932]{BigBookIntegrals}. Let $I_v$ and $K_v$ denote the modified Bessel functions defined in \cite[8.431, pg. 916]{BigBookIntegrals} and \cite[8.432, pg. 917]{BigBookIntegrals} respectively. So $I_v(u)$ is of exponential growth as $u\to \infty$ and $K_v(u)$ is bounded as $u\to \infty$. Since $W_{T,v}$ is of moderate growth as $w\to \infty$, there exists a constant $Y_{T,v}(h,s)\in \CC$ such that
\begin{equation}
f_{T,v}(h,w,s)=Y_{T,v}(h,s)K_v(|\beta_T(h,w,s)|). 
\end{equation}
Using the fact that $w\cdot \partial_w u=u$, we may apply \cite[8.486, pg 929]{BigBookIntegrals} to deduce
\begin{align*}
-(w\partial_w-v)K_v(u)=-(w\frac{\partial u }{\partial w}\frac{\partial}{\partial u}-v)K_v(u) 
=-(u\partial_{u}-v)K_v(u) =uK_{v+1}(u). 
\end{align*}
It follows that $\beta_T(w,s,h) Y_{T,v+1}(s,h)K_{v+1}(u)=Y_v(s,h)|\beta_T(w,s,h)|K_{v+1}(u)$.
As $K_v(u)$ is nowhere vanishing,
$Y_{T,v+1}(s,h)=Y_{T,v}(s,h)|\beta_T(w,s,h)|\beta_T(w,s,h)^{-1}$ and
\begin{equation}\label{expression-for-ftv}
f_{T,v}(h,w,s)=Y_{T,0}(h,s)\cdot \left(\frac{|\beta_T(h,w,s)|}{\beta_T(h,w,s)}\right)^v\cdot K_v(|\beta_T(h,w,s)|). 
\end{equation}
The next lemma will be used to show that the functions $Y_{T,0}(s,h)$ are constant in $h$. 
\begin{lemma}
\label{beta-vanishes}
If $k=1,\ldots, n-1$ then as functions on $M$ we have 
$$
[u_2\otimes \overline{v_k}]^+\cdot \beta_T\equiv 0 \qquad \hbox{and} \qquad [u_2\otimes \overline{v_k}]^-\cdot \overline{\beta_T}\equiv 0.
$$
\end{lemma}
\begin{proof} To give an indication of the method of proof we explain why $[u_2\otimes \overline{v_k}]^+\cdot \beta_T\equiv 0$.  By definition of the exponential, if $t\in \RR$ then
$u_2\cdot \exp(-t(u_2\otimes \overline{v_k}-v_k\otimes \overline{u_2}))=\cosh(t)u_2-\sinh(t)v_k$ and $u_2\cdot \exp(-it(u_2\otimes \overline{v_k}+v_k\otimes \overline{u_2}))=\cosh(t)u_2+\sinh(it)v_k$. It follows that if $m=(h,z)\in M$ then $[u_2\otimes \overline{v_k}]^+\cdot \beta_T(m)$ equals
$$
\frac{\sqrt{2}}{4}\cdot \frac{d}{dt}\left(\langle \cosh(t)u_2-\sinh(t)v_k, zT\cdot h\rangle-\sqrt{-1} \langle \cosh(t)u_2+\sinh(it)v_k, zT\cdot h\rangle \right)\rvert_{t=0}. 
$$
 Simplifying the above expression yields $[u_2\otimes \overline{v_k}]^+\cdot \beta_T(m)=0$ which completes the proof.
\end{proof}
\begin{prop}
\label{Proposition-a-large-portion}
If $\langle T,T\rangle <0$ then $\cW_T\equiv 0$, otherwise there exists a constant $Y_0\in \CC$ such that if $m=(h,z)\in M$ and $\beta_T(m)\neq 0$ then 
$$
\cW_T(m)=Y_{T,0}\cdot \sum_{-\ell\leq v\leq \ell}|z|^{2\ell+2}\left(\frac{|\beta_T(m)|}{\beta_T(m)}\right)^vK_v(|\beta_T(m)|)[u_1^{\ell-v}][u_2^{\ell+v}].
$$
\end{prop}

Before we prove Proposition~\ref{Proposition-a-large-portion}, we prove the following claim.

\begin{claim}
\label{claim1}
Given $(h,z)\in M$ such that $\beta_T(h,z)\neq 0$ define 
\begin{equation}\label{Wv-final-expression}
\cW_{T,v}(h,z)=|z|^{2(\ell+1)}\left(\frac{|\beta_T(h,z)|}{\beta_T(h,z)}\right)^vK_v(|\beta_T(h,z)|).
\end{equation}
Then $\cW_{T,v}$ satisfies the system of equations in Proposition \ref{System-Scalar-DEs}. 
\end{claim}
\begin{proof} 
We explain why (\ref{Wv-final-expression}) satisfies the system involving $[u_2\otimes \overline{v}_k]^+W_{T,v}(m)$. To simplify notation fix $k\in\{1,\ldots,n-1\}$ and write $dR=[u_2\otimes \overline{v_k}]^+$. A computation using Lemma \ref{beta-vanishes} and \cite[8.486, pg. 929]{BigBookIntegrals} implies 
\begin{equation}
\label{eqn1}
dR\left(\cW_{T,v}\right)=-\frac{|\beta_T|\cdot dR\left(|\beta_T|\right)}{\beta_T}\cdot \cW_{T,v-1}(|\beta_T|).
\end{equation}
Moreover, a computation similar to that given in the proof of Lemma \ref{beta-vanishes} yields 
\begin{equation}
\label{eqn2}
dR(\overline{\beta_T})=-\sqrt{2}\pi\overline{\langle v_k,zT\cdot h\rangle}.
\end{equation}
Since $dR(|\beta_T|)=\frac{\beta_T}{2|\beta_T|}dR(\overline{\beta_T})$, (\ref{eqn1}) and (\ref{eqn2}) imply the claimed statement.
\end{proof}

\begin{proof}[Proof of Proposition~\ref{Proposition-a-large-portion}]
So far we have shown that if $m=(h,w,s)\in M$ satisfies $\beta_T(h,w,s)\neq 0$ then there exists a constant $Y_{T,0}(h,s)\in \CC$ such that 
$$\cW_T(m)=Y_{T,0}(h,s)\cdot \sum_{-\ell\leq v\leq \ell}w^{2\ell+2}\left(\frac{|\beta_T(m)|}{\beta_T(m)}\right)^vK_v(|\beta_T(m)|)[u_1^{\ell-v}][u_2^{\ell+v}].
$$
We first show that $Y_{T,0}(h,s)$ is constant in $h$ and $s$. The fact that $\cW_T(gk)=\cW_T(g)\cdot k$ implies that $f_{T,v}(h,w,s)=e^{isv}f_{T,v}(h,w,0)$. Moreover $\beta_T(h,w,s)=e^{-isv}\beta_T(h,w,0)$. It follows from (\ref{expression-for-ftv}) that $Y_{T,0}(h,s)=Y_{T,0}(h,0)$. To show that $Y_{T,0}(h,0)$ is constant in $h$, it suffices to show that $[u_2\otimes \overline{v_k}]^{\pm}Y_{T,0}(h,0)\equiv 0$ for all $1\leq k < n$. Since we are assuming $\cD_{\ell}^{\pm}\cW_T\rvert_M\equiv 0$, this follows from Proposition \ref{System-Scalar-DEs} and Claim~\ref{claim1}.

It remains to show that $\cW_T\equiv 0$ whenever $\langle T,T\rangle<0$. Thus suppose $\langle T,T\rangle <0$. We know that there exists a constant $Y_0\in \CC$ such that if $m\in M$ satisfies $\beta_T(m)\neq 0$ then $\cW_T(m)=Y_0\cdot \sum_{-\ell\leq v \leq \ell} \cW_{T,v}(m)[u_1^{\ell-v}][u_2^{\ell+v}]$ with $\cW_{T,v}$ as defined in (\ref{Wv-final-expression}). Since $\langle T,T\rangle<0$ there exists $h'\in \U(V_0)$ such that the element  $m'=(h',1)\in M$ satisfies $\beta_T(m')=0$. The set $\{m\in M\colon \beta_T(m)\neq 0\}$ is Zariski open in $M$ and so there exists a sequence $\{m_i\in M\colon i\in \ZZ_{\geq 1}\}$ such that $m_i\to m'$ as $i\to \infty$ and $\beta_T(m_i)\neq 0$ for all $i\geq 1$. By assumption $\cW_T(m)$ is continuous at $m=m'$ and so 
\begin{equation}\label{limiting-over-i}
\cW_T(m')=Y_0\sum_{-\ell\leq v \leq \ell}\lim_{i\to \infty}\left( \cW_{T,v}(m_i)\right) [u_1^{\ell-v}][u_2^{\ell+v}].
\end{equation}
Writing $m_i=(h_i,z_i)$, formula (\ref{Wv-final-expression}) implies
$$
\lim_{i\to \infty}|\cW_{T,v}(m_i)|=\lim_{i\to \infty} |z_i|^{2(\ell+1)}K_v(|\beta_T(m_i)|).
$$ 
Since $m_i\to m'$ we know that $z_i\to 1$ as $i\to \infty$. Moreover, $K_v(u)$ has a pole at $u=0$ and so $K_v(|\beta_T(m_i)|)\to \infty$ as $i\to \infty$. It follows that $\lim_{i\to \infty}|\cW_{T,v}(m_i)|$ does not exist. Therefore the equality (\ref{limiting-over-i}) is only possible if $Y_{0}=0$. 
\end{proof}

\subsubsection{Verifying the Schmid Equations}
\label{subsubsection-verifying-schmid-equations} Recall the candidate solution $\cW_T\colon G\to \VV_{\ell}$ defined by (\ref{Main-Formula-Mult-One}). From \S \ref{subsubsection-establishing-the-candidate-solution} we know either $\cC_{N,\chi_{T}}^{\mathrm{md}}(G, \VV_{\ell})_{K_{\infty}}^{\cD_{\ell}=0}$ is zero, or it is spanned by the function $\cW_T$. Using Proposition \ref{System-Scalar-DEs} and Claim \ref{claim1}, one may check that $\cD_{\ell}^{\pm} \cW_T\rvert_M\equiv 0$. To prove Theorem \ref{Theorem-Multiplicity-At-Most-1} it remains to prove the lemma below. 
\begin{lemma}
    Assume $\langle T,T\rangle \geq 0$. Then $\cW_T$ is a smooth function of moderate growth satisfying the conditions $\cD_{\ell}^{\pm}\cW_T\equiv 0$.
\end{lemma}
\begin{proof} Since $\langle T,T\rangle \geq 0$, we have $\beta_T(m)\neq 0$ for all $m\in M$. It follows that (\ref{Main-Formula-Mult-One}) defines a smooth function on $G$. According to \cite[8.446, pg. 919]{BigBookIntegrals}, the leading order term in the Laurent expansion of $K_v(u)$ as $u\to 0^+$ is $u^{-v}$  and so (\ref{Main-Formula-Mult-One}) is of moderate growth. It remains to show that $\cD_{\ell}^{\pm}\cW_T\equiv 0$. As we have already explained, $\cD_{\ell}^{\pm}\cW_T\rvert_M\equiv 0$. Moreover, if $n\in N$ and $m\in M$ then $\cD_{\ell}^{\pm}\cW_T(nm)=\chi_T(n)\cD_{\ell}^{\pm}\cW_T(m)=0$. By the $K_{\infty}$ equivariance of $\cD_{\ell}^{\pm}$ it follows that $\cD_{\ell}^{\pm}\cW_T\equiv 0$.
\end{proof}
\section{Generalities on the Theta Correspondence}
\label{section-generalities-theta-correspondence}
In this section, we review the theory of theta correspondence for the dual pair $(\U(2,n), \U(1,1))$. Subsection \ref{subsection-The-metaplectic-group} describes several preliminaries on Weil representations of the dual pair $\U(2,n)\times \U(1,1)$. In subsection \ref{subsec:Fourier-Coefficients-of-the-theta-lift} we calculate the Fourier coefficient of the theta lift of a general cusp form $f$ from $\U(1,1)$ to $\U(2,n)$. Subsection \ref{subsection-non-vanishing} recalls a classical argument, dating back to \cite{PS83}, which shows that the theta lift of a cusp form from $\U(1,1)$ to $\U(2,n)$ is non-zero. Finally, in subsection \ref{subsec:The-theta-lifts-of-poincare-series} we review the definition of Poincar\'e series, and present a formal computation describing the theta lift from Poincar\'e series on $\U(1,1)$ to $\U(2,n)$.

\subsection{The Metaplectic Group and Weil representation}
\label{subsection-The-metaplectic-group}
Let $(\bfW,\langle \, , \,\rangle_{\bfW}) $ be a split skew-hermitian space over $E$ of signature (1,1). So $\bfW$ is a $2$-dimensional $E$-vector space. We write $\{w_+, w_-\}$ for an isotropic basic of $\bfW$ satisfying $\langle w_+, w_-\rangle_{\bfW}=1$. 
Let $\bfH=\U(\bfW)$ be the $\QQ$-rational unitary group associated to $(\bfW, \langle\cdot, \cdot \rangle_{\bfW})$ with the convention that $\bfH$ acts on the right of $\bfW$. Consider the $\QQ$-vector space $\mathbb{W}=\bfV\otimes_E \bfW$ endowed with the symplectic form
\begin{equation*}
    \ll v\otimes w, v^\prime\otimes w^\prime \gg := \text{tr}_{E/\mathbb{Q}}( \langle v, v^\prime \rangle \overline{\langle w, w^\prime\rangle_\bfW}), \, \, \mathrm{ for } \,\, v, v^\prime\in \bfV, w, w^\prime\in \bfW.
\end{equation*}
Through its natural action on $\bfV\otimes_E\bfW$, the group $\bfG\times \bfH$ is embedded in $\Sp(\WW)$ as a dual pair. Write 
$\mathrm{Mp}(\WW)(\AA)\to \Sp(\WW)(\AA)$ for the metaplectic $\CC^{\times}$-extension associated to $\psi$ and let $\omega_{\psi}$ denote the corresponding Weil representation (see for example \cite[\S 8]{Prasad93}). The cover $\mathrm{Mp}(\WW)(\AA)\to \Sp(\WW)(\AA)$ splits over the image of $\bfG(\AA)\times \bfH(\AA)$ in $\Sp(\WW)(\AA)$. To normalize such a splitting, let $\chi\colon E^{\times}\backslash \AA_E^{\times}\to \CC^1$ denote a character such that $\chi\rvert_{\AA^{\times}}$ coincides with the character associated to the extension $E/\QQ$ under class field theory. It follows from \cite[Proposition 3.1.1]{GeRo91} that the pair $(\psi,\chi)$ determines a splitting $s_{\psi,\chi}\colon \bfG(\AA)\times \bfH(\AA)\to \mathrm{Mp}(\WW)(\AA)$. 
The composition $\omega_{\psi, \chi}=\omega_{\psi}\circ s_{\psi,\chi}$ defines a Weil representation of $\bfG(\AA)\times \bfH(\AA)$. 

\subsubsection{The Schr\"odinger model}
Let $\mathbb{X}=\mathbf{V}\otimes w_+$ and $\mathbb{Y}=\mathbf{V}\otimes w_{-}$ so that $\mathbb{W}=\XX\oplus \YY$ is a decomposition into Lagrangian subspaces. The Weil representation $\omega_{\psi, \chi}$ of $\bfG(\AA)\times \bfH(\AA)$ admits a model in the Schwartz space $\mathcal{S}(\mathbb{X}(\AA))$. We write elements in $\bfH(\AA)$ as matrices relative to the basis $\{w_+,w_-\}$. Given $\phi\in \mathcal{S}( \XX(\AA))$, $x\in \bfV(\AA)$, $a\in \AA_E^{\times}$, $b\in \AA$, and $g\in \bfG(\AA)$, we have \cite[\S 1]{Ichino04}
\begin{equation}
\label{eqn:WeilRepSch}
\begin{split}
    \omega_{\psi, \chi}\left(1,\left( \begin{smallmatrix}
    a & \\
    &   \overline{a}^{-1}\end{smallmatrix}\right)\right)\phi( x \otimes w_+ ) & =\chi(a)^{2+n}|a|_{\AA_E}^{\frac{2+n}{2}} \phi(x \otimes aw_+ ),  \\
     \omega_{\psi, \chi}\left(1,\left( \begin{smallmatrix}
    1 & b\\
    & 1\end{smallmatrix}\right)\right)\phi(x\otimes w_+) & = \psi( \langle x, x\rangle b  )\phi(x\otimes w_+ ),  \\
   \omega_{\psi, \chi}\left(1,\left( \begin{smallmatrix}
     & 1\\
    -1 & \end{smallmatrix}\right)\right)\phi(x\otimes w_+ ) &= \gamma(\psi,\bfV) \cF \phi(x\otimes w_+)= \gamma(\psi,\bfV) \int\limits_{\bfV(\AA)} \phi(y \otimes w_+)\psi_E(  \langle x, y\rangle ) dy,   \\
    \omega_{\psi, \chi}(g,1)\phi(x\otimes w_+ ) &=\phi( xg  \otimes w_+), \\
\end{split}
\end{equation}
Here, $\gamma(\psi,\bfV)$ is the Weil index \cite[n$^{\circ}$14, pg. 161]{Weil1964},  and $dy$ denotes the Haar measure on $\mathbf{V}(\AA)$ relative to which $\cF$ is self-dual. For an algebraic group $R$ over $\mathbb{Q}$, we write $[R]$ to denote the adelic quotient $R(\QQ)\backslash R(\AA)$.
For a Schwartz function $\phi\in \mathcal{S}( \XX(\AA))$, $g\in \bfG(\AA)$, and $h\in \bfH(\AA)$ we define a theta function $\theta(\cdot, \cdot ; \phi)\colon [\bfG]\times [\bfH]\to \CC$ by
\begin{equation}
\label{defn-of-theta-fcn}
    \theta(g, h;\phi) =\sum_{\xi \in \mathbb{X}(\mathbb{Q})}\omega_{\psi, \chi}(g, h)\phi(\xi),
\end{equation}
which is of moderate growth in each variable. 
For a cusp form $f$ on $\bfH(\AA)$, the theta-lift $\theta(f,\phi)\colon [\bfG]\to \CC$ is defined as
\begin{equation}
\label{eq-theta-lift}
    \theta(f, \phi)(g)= \int\limits_{\bfH(\mathbb{Q})\backslash \bfH(\AA)}  \theta(g, h;\phi) f(h)dh.
\end{equation}
Since $f$ is cuspidal, the above integral converges absolutely.

\subsubsection{The mixed model}

Recall that we have a decomposition $\bfV=\bfU\oplus \bfV_0\oplus \bfU^\vee$. Then the decomposition $\mathbb{W}=\mathbb{X}_{\bfU}\oplus \mathbb{Y}_U$, with
\begin{equation*}
    \mathbb{X}_{\bfU}= \bfV_0\otimes w_+ \oplus b_1\otimes \bfW, \,\,  \mathbb{Y}_U=b_2 \otimes \bfW \oplus \bfV_0\otimes w_-,
\end{equation*}
is a Lagrangian decomposition of $\mathbb{W}$.
We define a partial Fourier transform
\begin{equation*}
\begin{split}
\mathcal{F}_\bfU:
\mathcal{S}((\bfV \otimes w_+)(\AA)) &\to \mathcal{S}((b_1\otimes \bfW)(\AA)) \otimes \mathcal{S}((\bfV_0\otimes w_+)(\AA))
\end{split}
\end{equation*}
by
\begin{equation}
\label{eq-Fourier-transform}
    \mathcal{F}_\bfU (\phi) (b_1\otimes x w_+, b_1\otimes y w_-, v\otimes w_+) =  \int\limits_{\AA_E} \phi(b_1\otimes x w_+, v\otimes w_+, b_2\otimes w_+ z)\psi_E(\langle zw_+,yw_-\rangle )dz,
\end{equation}
where  $x\in \AA_E$, $y\in \AA_E$, and $v\in \bfV_0(\AA)$. Then $\mathcal{F}_\bfU$ defines an isomorphism between the above two spaces, and this gives rise to the mixed model for the Weil representation $\omega=\omega_{\psi, \chi, \mathbb{X}_{\bfU}}$ of $\bfG(\AA)\times \bfH(\AA)$ on $\mathcal{S}((\bfV_0\otimes \mathbb{X})(\AA)) \otimes \mathcal{S}((\bfU\otimes \bfW)(\AA))$. 

We now write down some explicit formulas for the action of $\omega_{\psi, \chi, \mathbb{X}_{\bfU}}$ in this model. 
Recall that $\mathbf{P}=\mathbf{M}\ltimes\mathbf{N}$ is the Heisenberg parabolic subgroup, with $\bfM=\U(\bfV_0)\times \Res_{E/\QQ}\GG_{m,E}$. 
Note that $\bfN$ is 2-step nilpotent. An element $n\in \bfN$ fixes $\bfU$ pointwise. If $\{h_1,\ldots, h_n\}$ is an ordered basis of $\bfV_0$, then with respect to the basis $\{b_1, h_1,\ldots, h_n, b_2\}$, the center $\bfZ=[\bfN,\bfN]$ is given by the matrices 
\begin{equation*}
     \bfZ=\left\{ n(z_0)= \begin{pmatrix}
    1 &  & \\
    0 & I_n & \\
    z_0 &0 & 1
    \end{pmatrix}\in \bfN\colon \hbox{$z_0\in E$ such that $\overline{z_0}=-z_0$}\right\}.
\end{equation*}
Denote
\begin{equation*}
N_0= \left\{ n(x_0)=  \begin{pmatrix}
    1 &  &  \\
    \ast & I_n & \\
    -\frac{1}{2}\langle x_0, x_0\rangle& x_0& 1
    \end{pmatrix} \in \bfN\colon x_0\in \bfV_0\right\}.
\end{equation*}
In the notation of \eqref{n_0-abelianised}, $n(x_0)=\exp(b_1\otimes \overline{x_0}-x_0\otimes \overline{b_1})$. Note that $\bfN= \bfZ N_0$.
\\
\indent Given a vector $v_0\in \bfV_0(\AA)$, we define a character $\eta_{v_0}$ on $N_0(\AA)$ by
\begin{equation}
\label{defn-of-eta}
    \eta_{v_0}(n(x))=\psi (\frac{1}{2}\tr_{E/\mathbb{Q}} (\langle {-v_0},x\rangle ) ), \qquad x\in \mathbf{V}_0(\AA). 
\end{equation}
Write $\eta_{v_0,\infty}$ for the archimedean component of $\eta_{v_0}$. A short computation reveals that the character $\eta_{v_0,\infty}$ is related to the character $\chi_{T,\infty}$ of \eqref{defn-chiv0} by 
\begin{equation}
\label{relation-between-eta-chi}
\eta_{v_0,\infty}(n(x))=\chi_{-iv_0,\infty}(n(x))
\end{equation}
where $x\in V_0$. 
We have the following formulas.

\begin{lemma}
\label{lemma-mixed-model-formula}
Let $x\in \AA_E$, $y\in \AA_E$, $v\in \bfV_0(\AA)$, $\phi^\prime\in \mathcal{S}((b_1\otimes \bfW)(\AA)) \otimes \mathcal{S}((\bfV_0\otimes w_+)(\AA))$.
We have
\begin{equation}
\label{eq-mixed-model-formula1} 
\begin{split}
       \omega_{\psi, \chi, \mathbb{X}_{\bfU}} \left(\begin{pmatrix}
     a & \\ &  \overline{a}^{-1}
     \end{pmatrix}\right)  & \phi^\prime(b_1\otimes x w_+, b_1\otimes y w_-, v\otimes w_+) \\
    =&    \chi(a)^{2+n}|a|_{\mathbb{A}_E}^{\frac{n}{2}} \phi^\prime(b_1\otimes ax w_+ , b_1\otimes \overline{a}^{-1}y w_- , av\otimes w_+),
\end{split}
\end{equation}
and
\begin{equation}
\label{eq-mixed-model-formula2} 
\begin{split}
     \omega_{\psi, \chi, \mathbb{X}_{\bfU}} \left(\begin{pmatrix}
     1 & b \\ &  1
     \end{pmatrix}\right) & \phi^\prime(b_1\otimes x w_+, b_1\otimes y w_-, v\otimes w_+) \\
     =& \psi_E( \langle v, v\rangle b) \phi^\prime(b_1\otimes x w_+, b_1\otimes (y+xb)w_- , v\otimes w_+).
\end{split}
\end{equation}
For $n(z_0), n(x_0)\in \bfN(\AA)$, we have
\begin{equation}
\begin{split}
\label{eq-mixed-model-formula4}
    \omega_{\psi, \chi, \mathbb{X}_{\bfU}}(n(z_0)) & \phi^\prime(b_1\otimes x w_+, b_1\otimes y w_-, v\otimes w_+)\\
    =& \psi_E( -z_0 x \overline{y} )\phi^\prime(b_1\otimes x w_+, b_1\otimes y w_-, v\otimes w_+)
\end{split}
\end{equation}
and
\begin{equation}
\label{eq-mixed-model-formula5}
\begin{split}
    \omega_{\psi, \chi, \mathbb{X}_{\bfU}}(n(x_0)) & \phi^\prime(b_1\otimes x w_+, b_1\otimes y w_-, v\otimes w_+) \\
  =&\psi_E(-(x \cdot (-\frac{1}{2}x_0 {}^t\overline{x_0})+\langle v, x_0\rangle) \overline{y}) \phi^\prime(b_1\otimes x w_+, b_1\otimes y w_-, (v+x x_0)\otimes w_+).
\end{split}
\end{equation}
\end{lemma}

\begin{proof}
This can be checked by assuming $\phi^\prime=\mathcal{F}_\bfU(\phi)$ for $\phi\in \mathcal{S}((\bfV\otimes w_+)(\AA))$. See, for example, \cite[pp. 340-341]{Rallis1984}, \cite[Section 7.4]{GanIchino2016} and \cite[Proposition 2.1.1]{Pollack21b}, for similar statements. We omit the details.
\end{proof}

Given a Schwartz function $\phi\in \mathcal{S}( \XX(\AA))$, recall the theta function $\theta(\cdot, \cdot ; \phi)\colon [\bfG]\times [\bfH]\to \CC$ defined in \eqref{defn-of-theta-fcn}. Similarly, given a Schwartz function $\phi^\prime\in \mathcal{S}((b_1\otimes \bfW)(\AA)) \otimes \mathcal{S}((\bfV_0\otimes w_+)(\AA))$, one can define 
\begin{equation*}
    \theta(g, h;\phi^\prime) =\sum_{\xi \in \mathbb{X}_{\bfU}(\mathbb{Q})} \omega_{\psi, \chi, \mathbb{X}_{\bfU}}(g, h)\phi^\prime(\xi).
\end{equation*}
It is known that if $\phi^\prime=\mathcal{F}_\bfU(\phi)$, then 
\begin{equation*}
     \theta(g, h;\phi)= \theta(g, h;\phi^\prime).
\end{equation*}

In the rest of the paper, we will simply write $\omega$ to denote the Weil representation. The reader may determine which model of $\omega$ we are using based on the domain of the Schwartz data $\phi$. 

\subsection{Fourier Coefficients of the Theta Lift}
\label{subsec:Fourier-Coefficients-of-the-theta-lift}
 Let $\bfN_\mathbb{Y}=\left\{\begin{pmatrix} 1 & x\\ & 1\end{pmatrix}\in \bfH \right\}$ and $\bfM_\mathbb{Y}=\left\{\begin{pmatrix} a & \\ & \overline{a}^{-1}\end{pmatrix}\in \bfH \right\}$. 
Suppose $t\in \mathbb{Q}$ and let $\chi_t:[\bfN_\mathbb{Y}]\to \bbC^\times$ be defined by
\begin{equation*}
    \chi_{t} \begin{pmatrix} 1 & x\\ & 1\end{pmatrix} = \psi(tx).
\end{equation*}
For an automorphic form $\varphi$ on $\bfH(\AA)$, the $\chi_t$-Fourier coefficient of $f$ is
\begin{equation*}
    f_t(g)=f_{\chi_{t}}(g)=\int\limits_{[\bfN_\mathbb{Y}]}f(ng)\chi_t^{-1}(n)dn.
\end{equation*}
Given a character $\chi:[\bfN]\to \bbC^\times$ and a Schwartz function $\phi\in \cS(\XX(\AA))$ we define
\begin{equation*}
    \theta_\chi(g, h;\phi)=\int\limits_{[\bfN]} \theta(ng, h;\phi)\chi^{-1}(n)dn.
\end{equation*}
Let $f\colon [\bfH]\to \CC$ be a cusp form and recall the theta-lift $\theta(f;\phi)$ defined in \eqref{eq-theta-lift}. Then the $\chi$-Fourier coefficient of the theta lift $\theta(f; \phi)$ along $\bfN$ is given by $\theta(f, \phi)_\chi(g)= \int\limits_{[\bfH]} \theta_\chi(g, h;\phi) f(h)dh$. \\
\indent
In the case when $\langle v_0,v_0\rangle \neq 0$, the next proposition expresses the Fourier coefficient $\theta(f, \phi)_{\eta_{v_0}}$ as an integral transform of the Fourier coefficient $f_{-\langle v_0, v_0\rangle}$.

\begin{prop}
\label{prop-Fourier-coefficient-of-theta-lift}
Suppose $\theta(f, \phi)$ is the theta function
on $\bfG(\AA)$ associated to a cuspidal automorphic function $f$ on $\bfH(\AA)$ and a Schwartz function $\phi\in \mathcal{S}(\mathbb{X}_{\bfU}(\AA))$. Suppose $v_0$ is in $\bfV_0(\mathbb{Q})$ such that $\langle {v_0}, {v_0}\rangle\not=0$. Set $z_{v_0}=   b_1\otimes w_- +v_0\otimes w_+.$
Then
\begin{equation}
\label{eq-Fourier-coeff-of-theta-lift}
\theta(f, \phi)_{\eta_{v_0}}(g)=\int\limits_{\bfN_\mathbb{Y}(\AA)\backslash \bfH(\AA)} \omega(g, h)\phi(z_{v_0}) f_{-\langle v_0, v_0\rangle} (h) dh.
\end{equation}
\end{prop}

\begin{proof}

Without loss of generality we may assume $g=1$ and $\phi = \phi_\bfU \otimes \phi_{\bfV_0}$ with $\phi_\bfU \in  \mathcal{S}((b_1\otimes \bfW)(\AA))$ and $\phi_{\bfV_0}\in \mathcal{S}((\bfV_0\otimes w_+)(\AA))$. Taking the constant term of $\theta(1,h;\phi)$ along $\bfZ\subset \bfN$, we see that 
\begin{equation*}
\begin{split}
    \theta_\bfZ(1, h;\phi) &= \int\limits_{\bfZ(\mathbb{Q})\backslash \bfZ(\AA)} \theta(z, h;\phi)dz   = \int\limits_{\bfZ(\mathbb{Q})\backslash \bfZ(\AA)}  \sum_{\xi\in \mathbb{X}_{\bfU}(\mathbb{Q})} \omega(z, h)\phi(\xi) dz\\
    &= \int\limits_{\bfZ(\mathbb{Q})\backslash \bfZ(\AA)}  \sum\limits_{\substack{v\in \bfV_0(\mathbb{Q})\\ w=x w_++ y w_- \in \bfW(\mathbb{Q})}} \psi_E(-z x \overline{y})  \omega(1, h)\phi(v\otimes w_+ +b_1\otimes w) dz
\end{split}
\end{equation*}
where in the last equality we have used \eqref{eq-mixed-model-formula4}. Interchanging the order of summation and the integration we obtain
\begin{equation}
\begin{split}
 \theta_\bfZ(1, h;\phi) &=  \sum\limits_{\substack{v\in \bfV_0(\mathbb{Q})\\ w\in \bfW(\mathbb{Q}): \langle w, w\rangle_\bfW=0}}   \omega(1, h)\phi(v\otimes w_+ +b_1\otimes w) \\
 &= \theta_{\bfV_0}(1, h;\phi_{\bfV_0}) \left(\sum_{w\in \bfW(\mathbb{Q}): \langle w, w \rangle_\bfW=0} \omega(1, h)\phi_\bfU(b_1 \otimes w) \right)
 \label{eq-theta_Z-1}
\end{split}
\end{equation}
where $\theta_{\bfV_0}(1, h;\phi_{\bfV_0}):=\sum_{v\in \bfV_0(\mathbb{Q})} \left(\omega(h)  \phi_{\bfV_0}\right) (v \otimes w_+)$.
The non-zero isotropic vectors in $\bfW(\mathbb{Q})$ lie in a single $\bfH(\QQ)$ orbit. In fact
\begin{equation*}
    \{w\in \bfW(\mathbb{Q}): \langle w, w\rangle_\bfW=0 \} = \{0\}\cup w_- \cdot \bfN_\mathbb{Y}(\mathbb{Q})\backslash \bfH(\mathbb{Q}).
\end{equation*}
In the mixed model of $\omega$, the actions of $\bfN_\mathbb{Y}(\mathbb{Q})$ and $\bfM_\mathbb{Y}(\mathbb{Q})$ are given in Lemma~\ref{lemma-mixed-model-formula}. Moreover, a similar computation yields 
\begin{equation*}
     \omega \left(1, \begin{pmatrix}
     &1\\ -1
     \end{pmatrix}\right)  \phi_\bfU(b_1\otimes w) 
    =    \phi_\bfU(b_1\otimes w \begin{pmatrix}
     &1\\ -1
     \end{pmatrix} )
\end{equation*}
for any $w\in  \bfW(\AA)$. We conclude that for any $h_0\in \bfH(\QQ)$, $\omega(1,h_0)  \phi_\bfU(b_1\otimes w_-)=\phi_\bfU(b_1\otimes w_- h_0)$.
Therefore,  
\eqref{eq-theta_Z-1} can be rewritten as
\begin{equation*}
 \theta_{\bfV_0}(1, h;\phi_{\bfV_0}) \left(\omega(1,  h)\phi_\bfU(0) +  \sum\limits_{h_0\in \bfN_\mathbb{Y}(\mathbb{Q})\backslash \bfH(\mathbb{Q})} \omega(1,  h_0 h)\phi_\bfU(b_1\otimes w_-)\right).
\end{equation*}

Hence the $\eta_{v_0}$-Fourier coefficient $\theta_{\eta_{v_0}}(1, h;\phi)$ is given by
\begin{equation}
\label{eq-theta_N-1}
\begin{split}
\int\limits_{[N_0]}  &\theta_{\bfV_0}(n(x_0), h;\phi_{\bfV_0}) \omega(n(x_0),  h)\phi_\bfU(0)\eta_{-v_0}(n(x_0)) dx_0 \\
        +&   \int\limits_{[N_0]}  \theta_{\bfV_0}(n(x_0), h;\phi_{\bfV_0}) \sum\limits_{h_0\in \bfN_\mathbb{Y}(\mathbb{Q})\backslash \bfH(\mathbb{Q})} \omega(n(x_0),  h_0 h)\phi_\bfU(b_1\otimes w_-)\eta_{-v_0}(n(x_0)) dx_0.
\end{split}
\end{equation}
By \eqref{eq-mixed-model-formula5}, the product $\theta_{\bfV_0}(n(x_0), h;\phi_{\bfV_0}) \omega(n(x_0),  h)\phi_\bfU(0)$ is independent of $x_0$. Furthermore, since $v_0\neq 0$, the integral $\int_{[N_0]}\eta_{-v_0}(n(x_0))dx_0=0$. Hence, the first term in \eqref{eq-theta_N-1} is equal to $0$ and

\begin{align}
     \theta(f,&\phi)_{\eta_{v_0}}(1) =\int\limits_{[\bfH]} \theta_{\eta_{v_0}}(1, h;\phi) f(h)dh \nonumber\\
    =& \int\limits_{\bfN_\mathbb{Y}(\mathbb{Q})\backslash \bfH(\AA)}  \int\limits_{[N_0]} \theta_{\bfV_0}(n(x_0), h;\phi_{\bfV_0}) \omega(n(x_0),   h)\phi_\bfU(b_1\otimes w_-)\eta_{-v_0}(n(x_0)) dx_0 f(h)dh. \label{eq-theta_N-2}
\end{align}
Another application \eqref{eq-mixed-model-formula5} then yields an expression for $\theta(f, \phi)_{\eta_{v_0}}(1)$ of the form
\begin{equation}
\label{eqn-proof-FC-formula}
\int\limits_{\bfN_\mathbb{Y}(\mathbb{Q})\backslash \bfH(\AA)} 
      \sum_{v\in \bfV_0(\mathbb{Q})}\omega(1,h)(\phi_{\bfV_0}\otimes \phi_{\bfU})(v_0\otimes w_++b_1\otimes w_-)f(h) \int\limits_{[N_0]} \psi_E(-\langle v-v_0, x_0\rangle ) dx_0 dh.
\end{equation}
Since the pairing $\langle \, ,\,\rangle$ is non-degenerate on $\bfV_0$, \eqref{eqn-proof-FC-formula} simplifies to
\begin{equation*}
\begin{split}
   \theta(f, \phi)_{\eta_{v_0}}(1) =\int\limits_{\bfN_\mathbb{Y}(\mathbb{Q})\backslash \bfH(\AA)} \omega_{} (1, h)\phi(b_1\otimes w_- + v_0 \otimes w_+  )f(h)dh.
\end{split}    
\end{equation*}
Now we use \eqref{eq-mixed-model-formula2} to get
\begin{equation*}
\begin{split}
      &\theta(f, \phi)_{\eta_{v_0}}(1) \\
    =&     \int\limits_{\bfN_\mathbb{Y}(\AA)\backslash   
           \bfH(\AA)}\int\limits_{\bfN_\mathbb{Y}(\mathbb{Q})\backslash \bfN_\mathbb{Y}(\AA)} \omega_{} \left(1, \begin{pmatrix} 1 & x\\ & 1\end{pmatrix} h \right)\phi(b_1\otimes w_- + v_0 \otimes w_+ )f\left( \begin{pmatrix} 1 & x\\ & 1\end{pmatrix} h \right) dx dh \\
=&       \int\limits_{\bfN_\mathbb{Y}(\AA)\backslash   
           \bfH(\AA)}\omega_{} \left(1,  h \right)\phi(b_1\otimes w_- + v_0 \otimes w_+ )\int\limits_{\bfN_\mathbb{Y}(\mathbb{Q})\backslash \bfN_\mathbb{Y}(\AA)} f\left( \begin{pmatrix} 1 & x\\ & 1\end{pmatrix} h \right)\psi(\langle v_0, v_0\rangle x) dx dh\\
=&       \int\limits_{\bfN_\mathbb{Y}(\AA)\backslash   
           \bfH(\AA)}\omega_{} \left(1,  h \right)\phi(b_1\otimes w_- + v_0 \otimes w_+ )\int\limits_{\bfN_\mathbb{Y}(\mathbb{Q})\backslash \bfN_\mathbb{Y}(\AA)} f\left( \begin{pmatrix} 1 & x\\ & 1\end{pmatrix} h \right)\chi_{\langle v_0, v_0\rangle}(x) dx dh.
\end{split}
\end{equation*}
The inner integration in the above formula is equal to $f_{-\langle v_0, v_0\rangle}(h)$ as required.
\end{proof}
\subsection{Nonvanishing}
\label{subsection-non-vanishing}
In this section we sketch the proof that the theta lift of a nonzero cusp form on $\mathbf{H}(\mathbb{A})$ is nonzero. The argument closely follows an argument of Piatetski-Shapiro \cite{PS83}.

\begin{prop}
\label{prop-nonvanishing-theta-lift}
Let $\tau$ be a non-zero irreducible cuspidal automorphic representation of $\mathbf{H}(\mathbb{A})$. The theta lift $\theta(\tau, \psi)$ of $\tau$ to $\mathbf{G}(\mathbb{A})$ is non-zero.
\end{prop}

Before we prove Proposition~\ref{prop-nonvanishing-theta-lift}, we first state a lemma, whose proof is similar to that of \cite[Lemma 5.1]{PS83} and we omit it. 

\begin{lemma}
Let $v_0\in \bfV_0(\mathbb{Q})$ and $z_{v_0}$ be as in Proposition~\ref{prop-Fourier-coefficient-of-theta-lift}. If $W:\mathbf{H}(\mathbb{A})\to \mathbb{C}$ is a continuous function satisfying
$$W(nh)=\chi_{-\langle v_0, v_0\rangle}(n)W(h), \quad n\in \bfN_\mathbb{Y}(\AA), h\in \mathbf{H}(\mathbb{A}),$$
and for any Schwartz function $\phi\in \mathcal{S}((b_1\otimes \bfW)(\AA)) \otimes \mathcal{S}((\bfV_0\otimes w_+)(\AA))$, we have
$$\int_{\mathbf{N}_{\mathbf{Y}}(\mathbb{A})\backslash\mathbf{H}(\mathbb{A})}W(h)(\omega_{\psi}(h)\phi)(z_{v_{0}})dh=0$$
then $W\equiv 0$.
\label{lemma-non-vanishing-theta-lift-lemma}
\end{lemma}

\begin{proof}[Proof of Proposition~\ref{prop-nonvanishing-theta-lift}]
Let $f$ be a non-zero cusp form belonging to $\tau$. Since $f$ is necessarily generic, there exists a vector $v_0\in \mathbf{V}_0$ such that the $\chi_{-\langle v_0, v_0\rangle}$th Fourier coefficient of $f$ is non-zero. Moreover,
\begin{equation*}
f_{-\langle v_0, v_0\rangle} (nh)=\chi_{-\langle v_0, v_0\rangle}(n)f_{-\langle v_0, v_0\rangle} (h), \quad n\in \bfN_\mathbb{Y}(\AA), h\in \mathbf{H}(\mathbb{A}),
\end{equation*}
and so Lemma~\ref{lemma-non-vanishing-theta-lift-lemma} implies that there exists $\phi\in \mathcal{S}((b_1\otimes \bfW)(\AA)) \otimes \mathcal{S}((\bfV_0\otimes w_+)(\AA))$ such that 
\begin{equation*}
\int_{\mathbf{N}_{\mathbf{Y}}(\mathbb{A})\backslash\mathbf{H}(\mathbb{A})} f_{-\langle v_0, v_0\rangle} (h)(\omega_{\psi}(h)\phi)(z_{v_{0}})dh \not=0.
\end{equation*}
Now we use \eqref{eq-Fourier-coeff-of-theta-lift} to conclude that $\theta(f, \phi)_{\eta_{v_0}}(1)\not=0$, as desired. 
\end{proof}
\subsection{The theta lifts of Poincar\'e Series}
\label{subsec:The-theta-lifts-of-poincare-series}
Let $t\in \mathbb{Q}$, and let $\mu_t:\bfH(\AA)\to \CC$ be a function which satisfies $\mu_t(n(x)g)=\psi(tx)\mu_t(g)$ for $n(x)=\begin{pmatrix}1 & x\\ & 1 \end{pmatrix}\in \bfN_\mathbb{Y}(\AA)$. Associated to $\mu_t$, one can define a Poincar\'e series
\begin{equation*}
    P(h; \mu_t)= \sum_{\gamma\in \bfN_\mathbb{Y}(\mathbb{Q})\backslash \bfH(\mathbb{Q})} \mu_t(\gamma h),
\end{equation*}
which is an automorphic function when this sum converges absolutely. The following result computes the theta lift $P(h;\mu_t)$. 

\begin{lemma}
\label{lemma-theta-lifting-of-Poincare-series}
Suppose that the sum defining $P(h;\mu_t)$ converges absolutely to a cuspidal automorphic form on $\bfH(\AA)$. Let $\phi\in \mathcal{S}(\mathbb{X}(\AA))$. Suppose that
\begin{equation}
\label{eq-lemma-theta-lifting-of-Poincare-series-1}
    \displaystyle{\sum_{ v\in \bfV(\mathbb{Q})}} \int_{\bfN_\mathbb{Y}(\QQ)\backslash \bfH(\AA)} |\mu_t(h)| |   \omega(g, h)\phi(v\otimes w_+)| dh <\infty.
\end{equation}
Then we have
\begin{equation}
\label{formal-poincare-lift}
    \theta(P(\cdot; \mu_t);\phi)(g) = \sum_{v\in \bfV(\mathbb{Q}): \langle v, v \rangle =-t} \int_{\bfN_\mathbb{Y}(\AA)\backslash \bfH(\AA)} \mu_t(h)   \omega(g, h)\phi(v\otimes w_+) dh.
\end{equation}
\end{lemma}
\begin{proof}
Since $P(h;\mu_t)$  is cuspidal, the integral defining the $\theta$-lift $\theta(P(\cdot; \mu_t);\phi)(g)$ converges absolutely. Then 
\begin{equation*}
\begin{split}
    \theta(P(\cdot; \mu_t);\phi)(g) &=\int\limits_{[\bfH]} \theta(g,h;\phi)P(h;\mu_t)dh  \\
    &= \int\limits_{\bfN_\bfY(\mathbb{Q})\backslash \bfH(\AA)} \theta(g, h;\phi)\mu_t(h)dh\\
    &= \int\limits_{\bfN_\bfY(\AA)\backslash \bfH(\AA)} \mu_t(h) \int\limits_{[\bfN_\bfY]} \theta(g, nh;\phi) \psi_t(n) dn dh.
\end{split}
\end{equation*}
The finiteness assumption guarantees that the above formal computations are justified. 
Notice that for $n=n(x)=\begin{pmatrix} 1 & x \\ & 1\end{pmatrix}\in \bfN_\bfY(\AA)$, we have
\begin{equation*}
\begin{split}
    \theta(g, n(x)h;\phi) =\sum_{v\in \bfV(\mathbb{Q})}\omega(g, n(x)h)\phi(v\otimes w_+)  =  \sum_{v\in \bfV(\mathbb{Q})} \psi_E( \langle v, v\rangle x) \omega(g, h)\phi(v\otimes w_+).
\end{split}
\end{equation*}
Thus
\begin{equation*}
\begin{split}
    \int\limits_{[\bfN_\bfY]} \theta(g, nh;\phi) \psi_t(n) dn = \sum_{v\in \bfV(\mathbb{Q}): \langle v, v \rangle =-t}    \omega(g, h)\phi(v\otimes w_+) .
\end{split}
\end{equation*}
The result follows.
\end{proof}

\section{Theta Lifts of Holomorphic Poincar\'e Series to $\mathbf{G}$}
\label{section-theta-lifting}
In this section we define the adelic versions of the quaternionic modular forms $\theta(\xi_f,\phi_0)$ of Theorem \ref{thm2-intro}. These modular forms are obtained as the theta lifts of anti-holomorphic modular forms on $\U(1,1)$ with a special choice of archimedean test data. This choice of test data is explicitly defined in \ref{subsec:archimedean-test-data}. In subsection \ref{subsection:Quaternionic-Poincare-Lifts-on-G} we make an archimedean computation of $\theta(\xi_f,\phi_0)$ in the case when $\xi_f$ corresponds to a holomorphic Poincar\'e series on $\U(1,1)$. Based on these computations we are able to deduce Theorem \ref{thm2-intro} modulo the statement that $\theta(\xi_f,\phi_0)$ satisfies the differential equations $D_{\ell}^{\pm}\theta(\xi_f,\phi_0)\equiv 0$. This is established in Section~\ref{section-Fourier-coefficients-of-theta-lift} as Theorem \ref{thm-quaternionicity-B_l,v}.
\subsection{Holomorphic Modular Forms on $\mathbf{H}$}
Recall that $\bfH=\U(\bfW)$ denotes the quasi-split unitary group attached to the skew Hermitian space $\bfW=E\linspan\{w_+,w_-\}$. We write elements of $\bfH$ as matrices relative to the basis $\{w_+,w_-\}$ with the convention that $\bfH$ acts on the right of $\bfW$.
Let $K_{\infty}'$ be a the maximal compact subgroup of $\mathbf{H}(\RR)$ with an isomorphism 
$$
K_{\infty}'\cong \left\{z\begin{pmatrix}\cos\theta & -\sin\theta\\ \sin\theta & \cos\theta\end{pmatrix}\colon z\in \CC^1, \theta\in [0,2\pi)\right\}.
$$ 
Let $\mathcal{H}_{1,1}=\left\{ z\in \mathbb{C}: \mathrm{Im}(z)>0 \right\}$ be the hermitian upper half space. Then $h=\begin{pmatrix} a & b\\ c&d\end{pmatrix}\in \mathbf{H}(\RR)$ acts on $z\in \mathcal{H}_{1,1}$ by the usual linear fractional transformation. 
The automorphy factor associated to $h=\begin{pmatrix} a & b\\ c&d\end{pmatrix}\in \mathbf{H}(\RR)$ and $z\in \mathcal{H}_{1,1}$ is
\begin{equation*}
    j(h, z)=cz+d.
\end{equation*}
Given $f$ a cuspidal automorphic form on $\mathbf{H}(\AA)$, we say that $f$ is associated to a holomorphic modular form of weight $k$ if for each $h_{\mathrm{fin}}\in \bfH(\AA_{\mathrm{fin}})$, the function $h_{\infty}\mapsto j(h_{\infty},i)^{k}f(h_{\infty}h_{\mathrm{fin}})$ descends to a holomorphic function on $\cH_{1,1}$. It follows that if $f$ is associated to a holomorphic modular form on $\bfH$, then the Fourier expansion of $f$ takes the form 
\begin{equation*}
f(h_\fin h_\infty)=\sum_{t>0} a_f(t)(h_\fin)j(h_\infty, i)^{-N} e^{2\pi i t(i \cdot h_\infty)}
\end{equation*}
where $a_f(t): \mathbf{H}(\AA_\fin)\to \mathbb{C}$ is specific locally constant function depending on $f$ and $t$. By a result of Shimura, the Fourier coefficients $a_f(t)$ endow the space of modular forms on $\bfH$ with an algebraic structure. More precisely, if $K_{\mathrm{fin}}\leq \bfH(\AA_{\mathrm{fin}})$ is a level subgroup, then there exists a number field $L/\QQ$ such that the space of weight $k$ modular forms on $\bfH$ of level $K_f$ admits a basis consisting of forms $f$ for which the Fourier coefficients $a_f(t)$ takes values in $L$ \cite{Shimura1975, Shimura1978, Harris86}. 
\subsection{Archimedean Test Data}
\label{subsec:archimedean-test-data}
Our current goal is to use Lemma \ref{lemma-theta-lifting-of-Poincare-series} as a means of constructing quaternionic modular forms on $\bfG$. Notice that if the test function $\phi$ and the inducing section $\mu$ are factorizable, then the integral appearing inside the summand of expression \eqref{formal-poincare-lift} is Eulerian. In this section, we describe a choice of test data $\phi_{\infty}\in \cS(\XX(\RR))$ for which we can explicitly compute the archimedean component of this Eulerian integral. To begin, let $(\ell, n)$ be a pair of integers satisfying either $\ell\geq n$ or $1\leq \lfloor \frac{n-1}{2}\rfloor\leq \ell<n$.
Recall the orthogonal decomposition $V=V_{2}^{+}\oplus V_{n}^{-}$ given in (\ref{max-compact-identified}). 
For $v\in V$, denote $\|v\|:=\sqrt{\langle v,v\cdot \iota\rangle }$ (recall that $\iota$ is the element in $K_\infty$ which acts as identity on $V_2^+$ and acts as negative identity on $V_n^-$) and let $(v, w):=\langle v, w\cdot \iota \rangle$ be the positive definite hermitian form on $V$. Recall $\{u_1, u_2\}$ is an orthonormal basis of $V_2^+$. As a $\U(V_2^+)$-representation, $\overline{V_{2}^{+}} \simeq V_{2}^{+}\otimes \det {}_{\U(V_2^+)}^{-1}$ via the map $\delta$ given by $\overline{u}_2\mapsto u_1$, $\overline{u}_1\mapsto -u_2$. We thus obtain a map
\begin{equation*}
P_K:\mathrm{Sym}^{\ell}V_{2}^{+}\otimes \mathrm{Sym}^{\ell}\overline{V_{2}^{+}} \simeq \mathrm{Sym}^{\ell}V_{2}^{+}\otimes \mathrm{Sym}^{\ell}V_{2}^{+}\otimes \det {}_{\U(V_2^+)}^{-\ell}\to \VV_{\ell}
\end{equation*}
which is equivariant under $K_{\infty}= \U(V_2^+)\times \U(V_n^-)$. Here the last map is given by multiplication. 
Denote by $\mathrm{pr}_{2}$ the projection from $V$ to $V_{2}^{+}$, which is equivariant under $K_{\infty}$. Note that if $v\in V$ with $\langle v, v\rangle>0$, then $\Vert \mathrm{pr}_{2}(v)\Vert>0$. For $v\in V$, define 
\begin{equation*}
    Q_{\ell}\colon V\to \VV_{\ell}, \quad Q_{\ell}(v):=P_K(\mathrm{pr}_{2}(v)^{\ell}\otimes \mathrm{pr}_{2}(\overline{v})^{\ell}),
\end{equation*}
and set
\begin{equation}
    \phi_\infty(v\otimes w_+):=Q_{\ell}(v) e^{-2\pi \|v\|^{2}}.
\label{eq-phi_infty}
\end{equation}
\begin{lemma}
\label{lemma-Weil-repn-archimedean}
For $k_\theta=\begin{pmatrix}\cos\theta & -\sin\theta\\ \sin\theta & \cos\theta\end{pmatrix}\in K_\infty^\prime$, we have
\begin{equation}
\label{equation-Weil-repn-archimedean}
\omega_\infty(k_\theta)\phi_\infty(v\otimes w_+)= j(k_\theta,i)^{-2\ell-2+n}\phi_\infty(v\otimes w_+).
\end{equation}
\end{lemma}

\begin{proof}
To prove the lemma, it suffices to compute the differential of the action of $K_\infty^\prime$ on $\phi_\infty(v\otimes w_+)$. The proof is standard (see for example \cite[Lemma 2.5.14 and Proposition 2.5.15]{LionVergne1980} for a proof of the analogous statement in the case of $\mathrm{SL}_2$). We omit the details.
\end{proof}
\subsection{Theta lifts of Poincar\'e series to $\mathbf{G}$}
\label{subsection:Quaternionic-Poincare-Lifts-on-G}
Fix $t\in \QQ_{>0}$ and let $\ell$ and $n$ be as in subsection \ref{subsec:archimedean-test-data}.
We are ready to explicate the archimedean integral appearing in the theta lifts of Poincar\'e series. As stated in Theorem \ref{thm2-intro}, we intend to theta lift weight $N=2\ell+2-n$ anti-holomorphic modular forms on $\bfH$ to obtain weight $\ell$ quaternionic modular forms on $\bfG$. As such, we take the archimedean component of the inducing section in the Poincar\'e series to be
\begin{equation}
\label{eq-archimedean_section}
    \mu_{-t,\infty}(h)=\overline{\det(h)^{\ell+2}j(h,i)^{-N}e^{2\pi i t(i \cdot h)}}.
\end{equation}
Suppose $v\in V$ with $\langle v, v\rangle=t$. With $\phi_{\infty}$ as defined in subsection \ref{subsec:archimedean-test-data} and $\mu_{-t,\infty}$ as above, Lemma \ref{formal-poincare-lift} outputs the archimedean integral
\begin{equation}
\label{equation-archimedean-local-integral}
I_{\infty}(v; t):=\int_{\bfN_\mathbb{Y}(\RR)\backslash \bfH(\RR)} \mu_{-t,\infty}(h)\omega_{\infty}(g,h)\phi_\infty(v\otimes w_+)dh.
\end{equation}
\begin{prop}
\label{prop-archimedean-integral}
Let $v$ be in $V$ with $\langle v, v\rangle =t>0$. Then there exists $C\in \CC^{\times}$ (independent of $v$) such that 
\begin{equation*}
I_{\infty}(v; t)=C\cdot B_{\ell,v}(g)
\end{equation*}
where $B_{\ell,v}(g)=\dfrac{Q_{\ell}(vg)}{\Vert\mathrm{pr}_{2}(vg)\Vert^{4\ell+2}}$.
\end{prop}

\begin{proof}
Without loss of generality we may assume $g=1$.
Then writing $\delta_{\bfP_\mathbb{Y}}$ for the modulus character of $\bfP_\mathbb{Y}(\mathbb{R})=\bfN_\mathbb{Y}(\RR) \bfM_\mathbb{Y}(\RR)$, the Iwasawa decomposition of $\mathbf{H}(\RR)$ implies 
\begin{equation*}
I_{\infty}(v; t)= \int_{\bfM_\mathbb{Y}(\RR)} \int_{0}^{2\pi} \delta_{\bfP_\mathbb{Y}}^{-1}(m) \mu_{-t, \infty}(mk_{\theta}) \omega_\infty(mk_{\theta})\phi_\infty(v\otimes w_+) d\theta dm.
\end{equation*}
For $m\in \bfM_\mathbb{Y}(\RR)$ and $k_\theta=\begin{pmatrix}\cos\theta & -\sin\theta\\ \sin\theta & \cos\theta\end{pmatrix}$ we have $\mu_{-t, \infty}(mk_{\theta})=j(k_\theta, i)^{N}\mu_{-t, \infty}(m)$. So by \eqref{equation-Weil-repn-archimedean} we conclude
\begin{equation*}
I_{\infty}(v; t)= 2\pi  \int_{\bfM_\mathbb{Y}(\RR)} \delta_{\bfP_\mathbb{Y}}^{-1}(m) \mu_{-t, \infty}(m) \omega_\infty(m)\phi_\infty(v\otimes w_+)  dm.
\end{equation*}
We write $m\in \bfM_\mathbb{Y}(\RR)$ as $m=\begin{pmatrix} a & \\  &\overline{a}^{-1}\end{pmatrix}$ with $a\in \CC^\times$. Using the archimedean analogue of the first formula of \eqref{eqn:WeilRepSch}, we compute that
\begin{equation*}
\begin{split}
\omega_{\infty}\left( \begin{pmatrix} a & \\  &\overline{a}^{-1}\end{pmatrix}\right)\phi_{\infty}(v\otimes w_+) &=\left(\frac{a}{|a|}\right)^{n+2}(|a|^2)^{\frac{2+n}{2}}|a|^{2\ell}e^{-2\pi|a|^2 \|v\|^2}Q_{\ell}(v)\\
&=a^{n+2}\vert a\vert^{2\ell}e^{-2\pi\vert a\vert^{2}\|v\|^{2}}Q_{\ell}(v).
\end{split}
\end{equation*}
On the other hand, by the construction of $\mu_{-t, \infty}$ we get that
\begin{equation*}
\mu_{-t, \infty}\left( \begin{pmatrix} a & \\  &\overline{a}^{-1}\end{pmatrix}\right)=\overline{\left(\frac{a}{\overline{a}}\right)^{\ell+2}\overline{a}^Ne^{2\pi i \langle v, v\rangle |a|^2i}}=\overline{a}^{\ell+2}a^{N-\ell-2}e^{-2\pi|a|^2 \langle v, v\rangle }.
\end{equation*}
Noticing $\delta_{\bfP_\mathbb{Y}}\left( \begin{pmatrix} a & \\  &\overline{a}^{-1}\end{pmatrix}\right)=|a|^2$ and $N=2\ell+2-n$, it follows that
\begin{equation*}
\begin{split}
I_{\infty}(v; t) &=  Q_\ell(v) \cdot 2\pi   \int_{\mathbb{C}^\times} |a|^{4\ell+2} e^{-2\pi |a|^2 (\langle v, v\rangle+\|v\|^2)}d^\times a.
\end{split}
\end{equation*}
Finally we make a change of variable $a \mapsto \frac{a}{\Vert\mathrm{pr}_{2}(v)\Vert}$ to arrive at
\begin{equation*}
\begin{split}
I_{\infty}(v; t) &=  \frac{Q_{\ell}(v)}{\Vert\mathrm{pr}_{2}(v)\Vert^{4\ell+2}} \cdot 2\pi   \int_{\mathbb{C}^\times} |a|^{4\ell+2} e^{-4\pi |a|^2 }d^\times a.
\end{split}
\end{equation*}
The integral $\int_{\mathbb{C}^\times} |a|^{4\ell+2} e^{-4\pi |a|^2 }d^\times a$ is non-zero and independent of $v$, thus completing the proof.
\end{proof}
Before applying Proposition \ref{prop-archimedean-integral} to the study of theta lifts of Poincar\'e series to $\mathbf{G}$, we must verify that the convergence condition \eqref{eq-lemma-theta-lifting-of-Poincare-series-1} is satisfied. This is achieved in the Lemma below. 
\begin{lemma}
\label{lemma-convergence-theta-lift}
Suppose $\phi_{\fin}\in \mathcal{S}(\mathbb{X}(\AA_{\mathrm{fin}}))$, and $\mu_{-t,\mathrm{fin}}\in \Ind_{\bfN_\mathbb{Y}(\AA_{\mathrm{fin}})}^{\bfH(\AA_{\mathrm{fin}})}(\chi_{-t,\mathrm{fin}})$. Set $\mu_{-t}=\mu_{-t,\infty}\cdot \mu_{-t,\mathrm{fin}}$. Then for $\ell>n+1$ the sum
\begin{equation*}
    \displaystyle{\sum_{ v\in \bfV(\mathbb{Q})}} \int_{\bfN_\mathbb{Y}(\QQ)\backslash \bfH(\AA)} |\mu_{-t}(h)| |   \omega(g, h)(\phi_{\mathrm{fin}}\otimes\phi_{\infty})(v\otimes w_+) dh 
\end{equation*}
is finite.
\end{lemma}

\begin{proof}
By the same manipulations as in the proof of Proposition~\ref{prop-archimedean-integral}, the quantity in the lemma is bounded by a constant times
\begin{equation*}
    \displaystyle{\sum_{\tiny v\in \Lambda}}  \| Q_{\ell}(v)\|     \int_{\mathbb{C}^\times} |a|^{4\ell+2} e^{-2\pi |a|^2 (t +\|v\|^2)}d^\times a
\end{equation*}
for some lattice $\Lambda$ in $\bfV(\mathbb{R})$. Note that $t$ is fixed, and both $t$ and $\| v\|$ are non-negative. Also note that for any fixed $\alpha\in \mathbb{R}_{>0}$, there exists a positive constant $C_\alpha$ depending only on $\alpha$ such that
\begin{equation*}
e^{-b}\le C_\alpha b^{-\alpha} \quad \text{ for any $b>0$}.	
\end{equation*}
Then 
\begin{equation*}
\begin{split}
	\int_{\mathbb{C}^\times} |a|^{4\ell+2} e^{-2\pi |a|^2 (t +\|v\|^2)}d^\times a  &= \int_{\mathbb{C}^\times} |a|^{4\ell+2} e^{-2\pi |a|^2 t} e^{ -2\pi |a|^2\|v\|^2}d^\times a    \\
	&\le  \int_{\mathbb{C}^\times} |a|^{4\ell+2} e^{-2\pi |a|^2 t} C_{\alpha} (2\pi |a|^2\|v\|^2)^{-\alpha} d^\times a     \\
	&= C_{\alpha} (2\pi)^{-\alpha}  \cdot \frac{1}{\|v\|^{2\alpha}}   \int_{\mathbb{C}^\times} |a|^{4\ell+2-2\alpha} e^{-2\pi |a|^2 t}  d^\times a .  
\end{split}
\end{equation*}
In order that the integral in the last row converges, we need to take $\alpha$ with
\begin{equation}
\label{eq-lemma-finiteness-1}
4\ell+2-2\alpha>0,
\end{equation}
which is due to the convergence range of the integral expression of the Gamma function. 
On the other hand, we must also check the convergence of 
\begin{equation}
\label{eq-lemma-finiteness-2}
    \sum_{v\in \Lambda} \frac{ \| Q_{\ell}(v)\| }{\|v\| ^{2\alpha}}.
\end{equation}
Since 
$$\|Q_{\ell}(v)\|=\|P_K(\mathrm{pr}_2(v)^{\ell}\otimes\mathrm{pr}_2(\overline{v})^{\ell})\|=\|\mathrm{pr}_2(v)\|^{2\ell} \le \|v\|^{2\ell},$$ 
the sum~\eqref{eq-lemma-finiteness-2} is bounded by 
\begin{equation}
\label{eq-lemma-finiteness-3}
    \sum_{v\in \Lambda} \frac{1}{\|v\| ^{2\alpha-2\ell}}.
\end{equation}
In order that the sum~\eqref{eq-lemma-finiteness-3} converges, we need to take $\alpha$ with
\begin{equation}
\label{eq-lemma-finiteness-4}
\alpha-\ell>n+2,	
\end{equation}
which is due to the convergence range of the Epstein zeta function attached to the positive definite hermitian form $\| \cdot\|$. Combining \eqref{eq-lemma-finiteness-1} and \eqref{eq-lemma-finiteness-4}, it follows that when $\ell>n+1$, the sum in the lemma is convergent. 
\end{proof}
Applying Lemma \ref{lemma-theta-lifting-of-Poincare-series} we have that $\theta(P(\cdot;\mu_{-t});\phi_{\mathrm{fin}}\otimes\phi_{\infty})$ is a non-zero constant times
$$
\sum_{v\in \bfV(\QQ)\colon \langle v,v\rangle =-t}\int_{\bfN_{\YY}(\AA_{\mathrm{fin}})\backslash \bfH(\AA_{\mathrm{fin}})}\mu_{-t,\mathrm{fin}}(h)\omega(g_{\mathrm{fin}},h)\phi_{\mathrm{fin}}(v\otimes w_+)dh\cdot B_{\ell,v}(g_{\infty}).
$$
In Theorem \ref{thm-quaternionicity-B_l,v} we will show that $\cD_{\ell}^{\pm} B_{\ell,v}\equiv 0$. As the above summation is absolutely convergent, Theorem \ref{thm-quaternionicity-B_l,v} implies that $D_{\ell}^{\pm}\theta(P(\cdot;\mu_{-t});\phi_{\mathrm{fin}}\otimes\phi_{\infty})\equiv 0$. Hence if $\ell>n+1$, then the condition $D_{\ell}^{\pm}\theta(P(\cdot;\mu_{-t});\phi_{\mathrm{fin}}\otimes\phi_{\infty})\equiv 0$ means that $\theta(P(\cdot;\mu_{-t});\phi_{\mathrm{fin}}\otimes\phi_{\infty})$ generates an automorphic representation which is quaternionic discrete series at infinity. It is also true that $\theta(P(\cdot;\mu_{-t});\phi_{\mathrm{fin}}\otimes\phi_{\infty})$ is square integrable (see \cite[Proposition 3.4.1]{GeRo91}) and so by \cite[Theorem 4.3]{Wallach84}, Theorem \ref{thm-quaternionicity-B_l,v} implies that $\theta(P(\cdot;\mu_{-t});\phi_{\mathrm{fin}}\otimes\phi_{\infty})$ is cuspidal. The Poincar\'e series span the space of cusp forms on $\bfH$. So combining Proposition \ref{prop-nonvanishing-theta-lift} with the argument above, we obtain a proof of Theorem \ref{thm2-intro} which is conditional on Theorem \ref{thm-quaternionicity-B_l,v}. The precise statement is as follows.
\begin{thm}
\label{prop-relation-of-Fourier-coefficients}
Let $ \ell>n+1$ and suppose $f$ is the automorphic function on $\mathbf{H}(\AA)$ corresponding to a weight $N=2\ell+2-n$ holomorphic modular form. Assume the central character $\varepsilon=\prod_{v\leq \infty} \varepsilon_v$ of $f$ satisfies $\varepsilon_{\infty}(z)=z^{n+2}$. Then there exists $\phi_{\mathrm{fin}}\in \mathcal{S}(\mathbb{X}(\AA_\fin))$ such that $\theta(\overline{f}, \phi_{\mathrm{fin}}\otimes\phi_\infty)$ is a non-zero cuspidal weight $\ell$ quaternionic modular form on $\bfG$. Moreover, the constant term $\theta(\overline{f}, \phi_{\mathrm{fin}}\otimes\phi_\infty)_{\mathbf{Z}}$ is non-zero with Fourier expansion 
\begin{equation*}
\theta(\overline{f}, \phi^\prime\otimes\phi_\infty)_{\bfZ}(g)=\sum_{\tiny T\in \mathbf{V}_0(\QQ)\colon \langle T,T\rangle>0}C_{T}\cdot a_{f}(T;\phi_{\mathrm{fin}};g_{\mathrm{fin}})\cW_{-iT}(g_{\infty}).
\end{equation*}
Here the coefficients $C_{T}\in \CC$ and 
\begin{equation}
\label{eq-Fourier-coefficient-nondegenerate}
a_{f}(T;\phi_{\mathrm{fin}};g_{\mathrm{fin}})=\int\limits_{\bfN_\mathbb{Y}(\AA_\fin)\backslash \bfH(\AA_\fin)} \omega(g_{\mathrm{fin}},h)\phi_{\mathrm{fin}}(b_1\otimes w_- +T\otimes w_+)a_{\overline{f}}(-\langle T, T\rangle)(h)dh.\end{equation}
\end{thm}
\begin{remark}
We remark that Theorem~\ref{prop-relation-of-Fourier-coefficients} is compatible with known results in the theory of archimedean theta correspondences, specifically \cite[Theorem 6.2]{Li1990}. To see this, let $\widetilde{U}(1,1)$ and $\widetilde{U}(2,n)$ denote the double covers of $\mathbf{H}(\RR)$ and $\mathbf{G(\RR)}$ respectively. Similarly, write $\tilde{K}_{\infty}'$ and $\tilde{K}_{\infty}$ for the maximal compact subgroups of $\widetilde{U}(1,1)$ and $\widetilde{U}(2,n)$ covering $K_{\infty}'$ and $K_{\infty}$ respectively. With $N$ as in Theorem~\ref{prop-relation-of-Fourier-coefficients}, let $\tilde{\pi}_N$ denote the discrete series representation on $\widetilde{\U}(1,1)$ such that the minimal $\tilde{K}_{\infty}'$-type of $\tilde{\pi}_N$ has highest weight $\tilde{\tau} =\left(\dfrac{N}{2}, \dfrac{-N}{2}\right)$. \\
\indent In (loc. cite), Li shows that the theta lifting $\theta(\tilde{\pi}_N^{\vee})$ of the contragredient of $\tilde{\pi}_N$ to $\widetilde{\U}(2,n)$ is given by the discrete series representation of $\widetilde{\U}(2,n)$ whose minimal $\tilde{K}_{\infty}$-type has highest weight 
\begin{align*}
\tilde{\mu}&=\left(\dfrac{N+n-2}{2}, \dfrac{-N-n+2}{2},0,\ldots,0\right)=(\ell, -\ell, 0, \ldots, 0). 
\end{align*}
As in \S \ref{subsection-The-metaplectic-group}, a choice of character $\chi_{\infty}\colon \CC^{\times}\to \U(1)$ satisfying $\chi_{\infty}\rvert_{\RR^{\times}}=\varepsilon_{\CC/\RR}^{n+2}$ induces a splitting of $\mathbf{H}(\RR)\times \mathbf{G}(\RR)$ into the metaplectic $\CC^{\times}$-extensions $\mathrm{Mp}(\WW)(\RR)$. The analysis of K-types above implies that, via this splitting, the representation $\tilde{\pi}_N$ pulls back to the representation of $\mathbf{H}(\RR)$ whose minimal $K_{\infty}'$-type has highest weight $\left(\dfrac{N+n+2}{2}, \dfrac{-N+n+2}{2}\right)$, and $\theta(\tilde{\pi}_N^{\vee})$ pulls back to the representation $\Pi_{\ell}$ of \S \ref{subsection-quaternionic-modular-forms-on-G} (see for example \cite[\S 2.3]{MR2401812}). 
\end{remark}
\section{Algebraicity in the Fourier Expansions on $\bfG$}
\label{section-Fourier-coefficients-of-theta-lift}
In this section we sketch the proof of Theorem \ref{thm-quaternionicity-B_l,v}, showing that $B_{\ell,v}$ is annihilated by the operators $\cD_{\ell}^{\pm}$. As a byproduct of Theorems \ref{Theorem-Multiplicity-At-Most-1} and \ref{thm-quaternionicity-B_l,v}, we obtain an integral representation for the generalized Whittaker function $\cW_T$ (see \eqref{intergral-rep-of-W_v}). In Theorem \ref{thm-algebraicity}, we apply this integral representation to show that if $f$ is a modular form on $\U(1,1)$ such that all of the Fourier coefficients of $f$ are algebraic numbers, then $\phi_{\mathrm{fin}}$ may be chosen so that $\theta(\overline{f}, \phi^\prime\otimes\phi_\infty)_{\bfZ}$ is non-zero and $C_T\cdot a_{f}(T;\phi_{\mathrm{fin}};g_{\mathrm{fin}})\in \overline{\QQ}$ for all $T\in \bfV_0(\QQ)$ and $g\in \bfG(\AA_{\mathrm{fin}})$. 
\subsection{A Quaternionic Function}

For $v\in V$ with $\langle v, v\rangle >0$, recall the function $B_{\ell, v}:\mathbf{H}(\RR)\to \VV_{\ell}$ 
\begin{equation*}
    B_{\ell,v}(g):=\frac{Q_{\ell}(vg)}{\Vert\mathrm{pr}_{2}(vg)\Vert^{4\ell+2}}.
\end{equation*}
The following result is a close analogue of \cite[Theorem 3.3.1]{Pollack21b}.

\begin{thm}
\label{thm-quaternionicity-B_l,v}
Suppose that $\ell\ge 2$. Then the function $B_{\ell, v}(g)$ is quaternionic, i.e., $\cD_{\ell}^{\pm} B_{\ell, v}(g)\equiv 0$.
\end{thm}

To prove Theorem~\ref{thm-quaternionicity-B_l,v}, we begin with the following general formula describing the action of $\fg$ on $B_{\ell, v}$. Given $X\in \fg_0$, which may be viewed as an endomorphism of $V$, we write $X(v)$ to denote the application of the endomorphism $X$ to the vector $v$. So
\begin{equation*}
    X(v)=\frac{d}{dt}\left(v\cdot \exp(tX)\right)\rvert_{t=0}.
\end{equation*}
Given $v\in V$, consider the function $z:G\to V$ defined by $g\mapsto v\cdot g$. Then $X\in \fg_0$ acts on $z$ under the right regular action $X\cdot z= -X(z)$. This formula remains valid for $X\in \fg$. We now give a preliminary lemma, which is a version of \cite[Lemma 3.3.2]{Pollack21b}. The proof is a direct computation. 

\begin{lemma}
\label{lemma-quaternionicity-B_l,v-1}
Let $g\in G$ and $X=X_1+iX_2\in \fg$ with $X_1, X_2\in \fg_0$. We define $X^{\ast}=X_1-iX_2$. Set $z=v\cdot g$, $p=\mathrm{pr}_2(z)$ and $\overline{p}=\delta(\mathrm{pr}_2(\overline{z}))$. Similarly we set $X(p)=\mathrm{pr}_2(X(z))$ and $X(\overline{p})=\delta(\mathrm{pr}_2(X(\overline{z})))$
so that $X\cdot p=-X(p)$ and $X\cdot \overline{p}=-X(\overline{p})$. 
Then 
\begin{equation*}
X\cdot B_{\ell,v}(g) 
= \frac{p^{\ell-1} \overline{p}^{\ell-1}\left( (4\ell+2)((X(p), p)+(p, X^{\ast}(p)))p\overline{p}-2\ell(X(p)\overline{p}+X(\overline{p})p)\|p\|^{2} \right)}{2\|p\|^{4\ell+4}}.
\end{equation*}
\end{lemma}
Recall that the set $\{[u_i\otimes \overline{v}_j]^{\pm} \colon 1\le i\le 2, 1\le j\le  n\}$, defined in \eqref{eq-mathfrakp-basis}, is a basis for $\fp^{\pm}$. To simplify notation, we denote $X_{ij}^{\pm}=[u_i\otimes v_j]^{\pm}$.
Then for $z\in V$, we have
$$
X^{\bullet}_{ij}(z)=
\begin{cases}
\langle z, v_j\rangle u_i &, \hbox{if $\bullet=+$}, \\ 
-\langle z, u_i\rangle v_j &, \hbox{if $\bullet =-$,}
\end{cases}\quad \hbox{and} \quad 
X^{\bullet}_{ij}(\overline{z})=
\begin{cases}
-\langle u_i, z\rangle \overline{v}_j, &\hbox{if $\bullet=+$,} \\ 
\langle v_j, z\rangle \overline{u}_i, & \hbox{if $\bullet=-$.}
\end{cases}
$$  
Moreover, by Lemma~\ref{lemma-quaternionicity-B_l,v-1}, we have
$$
X_{ij}^+\cdot B_{\ell,v}(g)=
\frac{p^{\ell-1}\overline{p}^{\ell-1}((4\ell+2)(X_{\gamma_{ij}}^+(p), p)p\overline{p}-2\ell X_{ij}^+(p)\overline{p}\|p\|^2)}{2\|p\|^{4\ell+4}}
$$ 
and 
$$
X_{ij}^-\cdot B_{\ell,v}(g)=
\frac{p^{\ell-1}\overline{p}^{\ell-1}((4\ell+2)(p, X_{\gamma_{ij}}^+(p))p\overline{p}-2\ell X_{ij}^-(\overline{p})p\|p\|^2)}{2\|p\|^{4\ell+4}}. 
$$
For each pair $(i, j)$ with $1\le i\le 2, 1\le j\le  n$, we have $K_\infty$-equivariant contractions
\begin{equation*}
\begin{split}
    \langle \cdot , X_{ij}^-\rangle \colon \Sym^2 V_2^+ \otimes \det{}_{\U(V_2^+)}^{-1} &\to (V_2^+\otimes \det{}_{\U(V_2^+)}^{-1})\boxtimes V_n^-\\
     f(u_1,u_2) &\mapsto \partial_{u_i}f(u_1,u_2)\boxtimes v_j  \\
\end{split}
\end{equation*}
and 
\begin{equation*}
\begin{split}
    \langle \cdot , X_{ij}^+\rangle \colon \Sym^2 V_2^+ \otimes \det{}_{\U(V_2^+)}^{-1} &\to (V_2^+\otimes \det{}_{\U(V_2^+)}^{-1})\boxtimes \overline{V_n^-}\\
     f(u_1,u_2) &\mapsto (-1)^i \partial_{u_{i+1}}f(u_1,u_2)\boxtimes \overline{v}_j .
\end{split}
\end{equation*}
Here the index $i$ is interpreted modulo 2. Then
$X_{ij}^{+}\cdot B_{\ell, v}(g)\otimes X_{ij}^{-}$ contracts to
\begin{equation*}
\frac{\ell\cdot p^{\ell-2}\overline{p}^{\ell-1}}{2\|z\|^{4\ell+4}}\left((4\ell+2) p(X_{ij}^+(p), p)\langle p\overline{p}, X_{ij}^-\rangle -2(\ell-1)\|p\|^2X_{ij}^+(p)\langle p\overline{p}, X_{ij}^-\rangle -2\|p\|^2p\langle X_{ij}^+(p)\overline{p}, X_{ij}^-\rangle\right), 
\end{equation*}
and the element $X_{ij}^{-}\cdot B_{\ell, v}(g)\otimes X_{ij}^{+}$ contracts to
$$
\frac{\ell\cdot p^{\ell-1}\overline{p}^{\ell-2}}{2\|p\|^{4\ell+4}}\left( (4\ell+2) \overline{p}(p,X_{ij}^+(p))\langle p\overline{p}, X_{ij}^+\rangle -2(\ell-1)\|p\|^2(X_{ij}^-(\overline{p})\langle p\overline{p}, X_{ij}^+\rangle -2\|p\|^2\overline{p}\langle X_{ij}^-(\overline{p})p, X_{ij}^+\rangle ) \right). 
$$
Thus Theorem~\ref{thm-quaternionicity-B_l,v} follows from Proposition \ref{prop-quaternionicity-B_l,v} below, whose proof is directly analogous to the proof of \cite[Proposition 3.3.3]{Pollack21b}. 
\begin{prop}
\label{prop-quaternionicity-B_l,v}
Let the notation be as above. For any $1\le j \le n$, we have
\begin{equation}
\label{eq-prop-quaternionicity-B_l,v-1}
\sum_{i=1}^2\left((4\ell+2) p(X_{ij}^+(p), p)\langle p\overline{p}, X_{ij}^-\rangle -2(\ell-1)\|p\|^2X_{ij}^+(p)\langle p\overline{p}, X_{ij}^-\rangle -2\|p\|^2p\langle X_{ij}^+(p)\overline{p}, X_{ij}^-\rangle\right)=0
\end{equation}
and
\begin{equation}
\label{eq-prop-quaternionicity-B_l,v-2}
\sum_{i=1}^2 \left( (4\ell+2) \overline{p}(p,X_{ij}^+(p))\langle p\overline{p}, X_{ij}^+\rangle -2(\ell-1)\|p\|^2(X_{ij}^-(\overline{p})\langle p\overline{p}, X_{ij}^+\rangle -2\|p\|^2\overline{p}\langle X_{ij}^-(\overline{p})p, X_{ij}^+\rangle ) \right)=0.
\end{equation}
\end{prop}

\subsection{The Fourier Transform of $A_\ell$}

Let $v_0\in V_0$ be such that $\langle v_0, v_0\rangle>0$. In this subsection, we study a Fourier transform of the function 
$$
A_{\ell}(v_0)=\frac{Q_{\ell}(v_0)}{\Vert\mathrm{pr}_{2}(v_0)\Vert^{4\ell+2}}=\frac{P_K(\mathrm{pr}_2(v_0)^{\ell}\otimes\mathrm{pr}_2(\overline{v_0})^{\ell}}{\Vert\mathrm{pr}_{2}(v_0)\Vert^{4\ell+2}}.
$$
Recall the family of constants $\{C_T\colon T\in \bfV_0(\QQ), \langle T,T\rangle >0\}$ introduced in Theorem \ref{prop-relation-of-Fourier-coefficients}.  Our results in this section will be used to show that if $T_1,T_2\in \mathbf{V}_0(\QQ)$ are such that $\langle T_1, T_1\rangle>0$ and $\langle T_2, T_2\rangle>0$, then $C_{T_1}=C_{T_2}$. Recall the character $\eta_{v_0,\infty}$ from \eqref{defn-of-eta}. The Fourier transform in question is defined as \begin{equation}\label{eq-Fourier-integral-A_l}
\cF_{v_0}A_{\ell}(g):=\int_{\Stab_{N}(v_0)\backslash N}A_{\ell}(v_0\cdot ng)\overline{\eta_{v_0, \infty}(n)}dn.
\end{equation}
The right cosets of $\Stab_N(v_0)$ in $N$ are represented by the elements of the one parameter subgroup $\CC\to N$,  $z\mapsto \exp(b_1\otimes \overline{zv_0}-zv_0\otimes \overline{b_1})$. We have that $v_0\cdot \exp(b_1\otimes \overline{zv_0}-zv_0\otimes \overline{b_1})=v_0-\langle v_0,zv_0\rangle b_1$. It follows that $\cF_{v_0}A_{\ell}(1)$ is a non-zero constant multiple of the integral
$$
J(v_0, \ell):=\int_{\CC}A_{\ell}(zb_1+v_0)\overline{\psi_{E,\infty}(z)}dz,
$$
where $dz$ denotes the double Lebesgue measure on $\CC$ and $\psi_{E,\infty}(z)=\psi_{\infty}(\frac{1}{2}\tr(z))$. Note that the quantity $A_{\ell}(zb_1+v_0)$ is insensitive to replacing $v_0$ with its projection onto the positive definite complex two plane $V_2^+$. Moreover, since $\langle v_0,v_0\rangle>0$, there exists a unique value $a\in \CC^{\times}$ such that the projection of $v_0$ onto $V_2^+$ is given by $\mathrm{pr}_2(v_0)=\overline{a}u_2$. Hence, rewriting $J(\overline{a}u_2,\ell)$ using the change of variable $z\mapsto \overline{a}z$, and applying the homogeneity properties of $A_{\ell}$, one obtains
\begin{equation}
\label{eq-Fourier-transform-Jprime}
J(\overline{a}u_2, \ell) =\frac{1}{|a|^{2\ell}}\int_{\CC}A_{\ell}(zb_1+u_2)\overline{\psi_{E,\infty}(\overline{a}z)}dz.
\end{equation}
The above expression is not identically $0$ as a function of $\overline{a}$ since it is a Fourier transform integral. Hence there exists $a'\in \CC$ such that $J(a'u_2, \ell)\neq 0$.
\begin{lemma}
\label{Lemma-Abs-Convergence-J-integral}
Suppose $\ell\ge 1$, $a\in \CC^{\times}$, and $g\in G$. Then the integral
\begin{equation}
\label{eq-Fourier-transform-A_l}
    \int_{\CC}A_{\ell}((zb_1+u_2)g)\overline{\psi_{E,\infty}(\overline{a} z)}dz
\end{equation}
converges absolutely.
\end{lemma}
\begin{proof}
To prove the absolute convergence of \eqref{eq-Fourier-transform-A_l}, it suffices to consider the case $g=1$.
Since $\|Q_{\ell}(v)\|=\|P_K(\mathrm{pr}_2(v)^{\ell}\otimes\mathrm{pr}_2(\overline{v})^{\ell})\|=\|\mathrm{pr}_2(v)\|^{2\ell}$, we have 
 $$
 \|A_{\ell}((zb_1+u_2)g)\|=\dfrac{1}{\|\mathrm{pr}_2((zb_1+u_2)g)\|^{2\ell+2}}.
 $$
Taking $g=1$ we have that 
\begin{align*}
    \|\mathrm{pr}_2(zb_1+u_2)\|
    =\left\|\sum_{i=1}^2\langle zb_1+u_2,u_i\rangle u_i\right\| 
    =\sqrt{\frac{|z|^2}{2}+1}. 
\end{align*}
Hence 
\begin{align*}
\int_{\CC}\left\|A_{\ell}((zb_1+u_2)) \overline{\psi_{E,\infty}(\overline{a}z)}\right\|dz
    &=
    \int_{\CC}\dfrac{1}{(|z|^2/2+1)^{\ell+1}}dz.
\end{align*}
This proves the absolute convergence of \eqref{eq-Fourier-transform-A_l} for $g=1$, and hence for general $g$.
\end{proof}
Lemma \ref{Lemma-Abs-Convergence-J-integral} implies that \eqref{eq-Fourier-integral-A_l} converges absolutely. Hence by Theorem \ref{thm-quaternionicity-B_l,v}, the function defined by \eqref{eq-Fourier-integral-A_l} is annihilated by the differential operators $D_{\ell}^{+}$ and $D_{\ell}^-$. Moreover, the integral \eqref{eq-Fourier-integral-A_l} is of moderate  growth and satisfies the same equivariance properties as the $\eta_{v_0, \infty}$-th generalized Whittaker function $\cW_{-iv_0}$ (see subsection~\ref{subsection-Generalized-Whittaker-Coefficients}). Hence by the multiplicity at most one statement in Theorem~\ref{Theorem-Multiplicity-At-Most-1}, there exists a constant $C_{v_0, \ell}\in \CC$ (see \eqref{relation-between-eta-chi})  such that  
\begin{equation}
\label{intergral-rep-of-W_v}
\int_{\CC}A_{\ell}((zb_1+v_0)g)\overline{\psi_{E,\infty}(z)}dz=
C_{v_0, \ell}\cdot \cW_{-iv_0}(g).
\end{equation}
Moreover, by the discussion preceding Lemma \ref{Lemma-Abs-Convergence-J-integral}, we may fix a choice of the vector $a'\in \CC$ so that $C_{a'u_2,\ell}\neq 0$. In fact, we have the following result.
\begin{lemma}
\label{main-technical-lemma}
    The constant $C_{u_2,\ell}\neq 0$, and for all $v_0\in V$ such that $\langle v_0,v_0\rangle>0$,
    \begin{equation}
\label{eqn-fourier-integral-A_ell}
\int_{\CC}A_{\ell}((zb_1+v_0)g)\overline{\psi_{E,\infty}(z)}dz=C_{u_2, \ell} \mathcal{W}_{-iv_0}(g).
\end{equation}
\end{lemma}
\begin{proof}
Let $a\in \CC^{\times}$ and set $M=\diag(a, I_{n}, \overline{a}^{-1})\in G$.  For ease of notation, denote $n(z)=\exp(b_1\otimes \overline{zu_2}-zu_2\otimes \overline{b_1})\in G$ for $z\in \CC$, so that $u_2\cdot n(z)=u_2-\overline{z}b_1$. Then
\begin{equation*}
\begin{split}
\int_{\CC}A_{\ell}((zb_1+u_2)M)\overline{\psi_{E,\infty}(z)}dz
&=\int_{\CC}A_{\ell}(u_2 \cdot MM^{-1}n(-\overline{z}) M)\overline{\psi_{E,\infty}(z)}dz \\
&= |a|^2 \int_{\CC}A_{\ell}(u_2 \cdot n( -\overline{z}) )\overline{\psi_{E,\infty}(az)}dz.
\end{split}
\end{equation*}
Combining this with \eqref{eq-Fourier-transform-Jprime}, we obtain 
\begin{equation*}
\begin{split}
J(au_2, \ell) &= \frac{1}{|a|^{2\ell+2}} \int_{\CC}A_{\ell}((zb_1+u_2)M)\overline{\psi_{E,\infty}(z)}dz.
\end{split}
\end{equation*}
Therefore, the above manipulation, together with \eqref{intergral-rep-of-W_v} gives $J(au_2, \ell)=|a|^{-2\ell-2}C_{u_2,\ell}\cW_{-iu_2}(M)$. Moreover, by definition of $J(au_2,\ell)$, \eqref{intergral-rep-of-W_v} implies $J(au_2,\ell)=C_{au_2, \ell}\mathcal{W}_{-iau_2}(1)$. So
\begin{equation}
    \label{equation-for-constants}
C_{au_2, \ell}\mathcal{W}_{-iau_2}(1)=|a|^{-2\ell-2}C_{u_2,\ell}\cW_{-iu_2}(M).
\end{equation}
Applying Theorem~\ref{Theorem-Multiplicity-At-Most-1} we calculate
\begin{equation*}
\begin{split}
\mathcal{W}_{-iu_2}(M)
=\sum_{-\ell\leq v\leq \ell}|a|^{2\ell+2}\left(\frac{|\beta_{-iau_2}(1,1)|}{\beta_{-iau_2}(1,1)}\right)^vK_v\left(|\beta_{-iau_2}(1,1)|\right) = |a|^{2\ell+2}\mathcal{W}_{-iau_2}(1). 
\end{split}
\end{equation*}
Hence, $\cW_{-iau_2}(1)=|a|^{-2\ell-2}\cW_{-iu_2}(M)$, and together with \eqref{equation-for-constants}, this implies that $C_{au_2,\ell}=C_{u_2,\ell}$. In particular, $C_{au_2,\ell}=C_{a'u_2,\ell}$ for all $a\in \CC^{\times}$, and thus $C_{u_2,\ell}$ is non-zero. Since $A_{\ell}(zb_1+v_0)$ is insensitive to replacing $v_0$ by $\mathrm{pr}_2(v_0)$, the equality \eqref{eqn-fourier-integral-A_ell} follows from \eqref{intergral-rep-of-W_v}.
\end{proof} 
\begin{prop}
\label{prop-Fourier-transform-A_l}
With notation as in Theorem \ref{prop-relation-of-Fourier-coefficients}, $C_T$ is non-zero and independent of $T$. 
\end{prop}
\begin{proof}
    By Proposition \ref{eq-Fourier-coeff-of-theta-lift}, we have 
\begin{equation*}
C_T= \cW_{-iT}(1)^{-1}\int_{\bfN_\mathbb{Y}(\RR)\backslash \bfH(\RR)} \omega_{\infty}(1,h) \mathcal{F}_{\bfU, \infty}(\phi_\infty)(z_{T})  \mu_{-t,\infty}(h) dh.
\end{equation*}
Applying Lemma~\ref{lemma-Weil-repn-archimedean} and the definition of the action of $\omega_\infty$ on $\mathcal{S}(\mathbb{X}(\RR))$, we have $\omega_{\infty}(1,k_\theta) \mathcal{F}_{\bfU, \infty}(\phi_\infty)=j(k_\theta, i)^{-N}\mathcal{F}_{\bfU, \infty}(\phi_\infty)$ for $k_\theta=\begin{pmatrix}\cos\theta & -\sin\theta\\ \sin\theta & \cos\theta\end{pmatrix}\in K_\infty^\prime$. Thus, by the Iwasawa decomposition $\mathbf{H}(\RR)=\bfN_\mathbb{Y}(\RR) \bfM_\mathbb{Y}(\RR)K_\infty^\prime$ and the fact that $\bfM_\mathbb{Y}(\RR) \cap K_\infty^\prime=\left\{ \begin{pmatrix} e^{it} & \\ & e^{it}\end{pmatrix}:t\in \mathbb{R} \right\}$, we have
\begin{equation*}
C_T= 2\pi\cW_{-iT}(1)^{-1} \int_{\bfM_\mathbb{Y}(\RR)} \delta_{\bfP_\mathbb{Y}}^{-1}( m)\omega_\infty(m)\mathcal{F}_{\bfU, \infty}(\phi_\infty)(z_{T})  \mu_{-t, \infty}(m)  dm.
\end{equation*}
By writing $m\in \bfM_\mathbb{Y}(\RR)$ as $m=m(a)=\begin{pmatrix} a & \\  &\overline{a}^{-1}\end{pmatrix}$ with $a\in \CC^\times$, we use the archimedean analogues of \eqref{eq-mixed-model-formula1} and \eqref{eq-Fourier-transform} to compute that
\begin{equation*}
\begin{split}
\omega_{\infty}\left( \begin{pmatrix} a & \\  &\overline{a}^{-1}\end{pmatrix}\right)\mathcal{F}_{\bfU, \infty}(\phi_\infty) (z_{T}) &=\left(\frac{a}{|a|}\right)^{n+2}(|a|^2)^{\frac{n}{2}} \mathcal{F}_{\bfU, \infty}(\phi_\infty) (b_1\otimes \overline{a}^{-1} w_-+ aT \otimes w_+)\\
&=a^{n+2}\vert a\vert^{-2}\int_\CC \phi_\infty(b_2\otimes z w_++aT\otimes w_+)\psi_{E,\infty}(za^{-1})dz \\
&= a^{n+2} \int_\CC \phi_\infty(b_2\otimes az w_++aT\otimes w_+)\psi_{E,\infty}(z)dz.
\end{split}
\end{equation*}
Thus, we obtain that $C_T$ is equal to 
\begin{equation*}
\begin{split}
& 2\pi\cW_{-iT}(1)^{-1} \int_{\CC^\times} \delta_{\bfP_\mathbb{Y}}^{-1}(m(a)) a^{n+2} \int_\CC \phi_\infty(b_2\otimes az w_++aT\otimes w_+)\psi_{E,\infty}(z)dz  \mu_{-t, \infty}(m(a))  d^\times a .
\end{split}
\end{equation*}
This double integral is absolutely convergent. Changing the order of integration gives
\begin{equation*}
\begin{split}
C_T=&\cW_{-iT}(1)^{-1} \int_{\CC} 2\pi \int_{\CC^\times} \delta_{\bfP_\mathbb{Y}}^{-1}(m(a)) a^{n+2}  \phi_\infty(b_2\otimes az w_+ +aT\otimes w_+) \mu_{-t, \infty}(m(a)) d^\times a \psi_{E,\infty}(z)dz \\
=&\cW_{-iT}(1)^{-1} \int_{\CC} \int_{\bfN_\mathbb{Y}(\RR)\backslash \bfH(\RR)} \omega_\infty(1,h)\phi_\infty( b_2\otimes z w_+ +T\otimes w_+) \mu_{-t, \infty}(h)dh \psi_{E,\infty}(z)dz.
\end{split}
\end{equation*}
The inner integral of the last formula is equal to the integral $I_\infty(zb_2+T;t)$ of \eqref{equation-archimedean-local-integral} evaluated at $g=1$. Hence by Proposition~\ref{prop-archimedean-integral}, there exists a constant $C\in \CC^{\times}$ (independent of $T$) such that 
$$
C_T=C\cW_{-iT}(1)^{-1}\cdot \int_{\CC}A_{\ell}(zb_2+T)\psi_{E,\infty}(z)dz.
$$  Applying Lemma \ref{main-technical-lemma} we conclude $C_T=C\cW_{-iT}(1)^{-1}\cdot C_{u_2, \ell}\mathcal{W}_{-iT}(1)=C\cdot C_{u_2,\ell}\neq 0$ as required.
\end{proof}
\subsection{Algebraicity of Fourier Coefficients}

We may now complete the proof of Theorem \ref{thm3-intro}. 
\begin{thm}
\label{thm-algebraicity}
Suppose $\ell>n+1$. Let $f$ be the automorphic function on $\mathbf{H}(\AA)$ associated to a weight $N=2\ell+2-n$ cuspidal holomorphic modular form $f$ on $\cH_{1,1}$. Assume $f$ has central character $\varepsilon=\prod_{v\leq \infty} \varepsilon_v$ where $\varepsilon_{\infty}(z)=\overline{z}^{n+2}$, and suppose that the functions $a_f(t)\colon \bfH(\AA_{\mathrm{fin}})\to \CC$ are valued in a single algebraic extension $L/\QQ$ for all $t>0$. Then there exists $\phi_\fin\in \mathcal{S}(\mathbb{X}(\AA_\fin))$ such that $\theta(f, \phi_\fin\otimes\phi_\infty)$ is a non-zero quaternionic cusp form on $\bfG$ with Fourier coefficients in $L(\mu_{\infty})$. Here $L(\mu_{\infty})/L$ denotes the extension obtained by adjoining all roots of unity to $L$.
\end{thm}

\begin{proof}
Recall that the non-degenerate Fourier coefficients of $\theta(f, \phi)$ are given by \eqref{eq-Fourier-coefficient-nondegenerate}. It suffices to show that the integral in \eqref{eq-Fourier-coefficient-nondegenerate} is a finite sum. As functions of $h\in \bfH(\AA_\fin)$, both $\omega(h)\phi(b_1\otimes w_- +v_0\otimes w_+)$ and $\overline{a_f(\langle v_0, v_0\rangle)(h)}$ are right invariant under a compact open subgroup $U_0$ of the maximal compact open subgroup $\bfH(\hat{\mathbb{Z}})$ of $\bfH(\AA_\fin)$. Using the Iwasawa decomposition, we obtain
\begin{equation}
\label{eq-Fourier-coefficient-nondegenerate-1}
\begin{split}
a(v_0;\phi;f) = \int\limits_{\bfM_\mathbb{Y}(\AA_\fin) } \int\limits_{\bfH(\hat{\mathbb{Z}})}    \omega(kh)\phi_\fin(b_1\otimes w_- +v_0\otimes w_+)\overline{a_f(\langle v_0, v_0\rangle)(kh)} dkdh .
\end{split}
\end{equation}
Let $\{k_1, \cdots, k_m\}$ be a set of representatives of $U_0\backslash \bfH(\hat{\mathbb{Z}})$. Then \eqref{eq-Fourier-coefficient-nondegenerate-1} 
becomes
\begin{equation*}
\begin{split}   
a(v_0;\phi;f) = \sum_{j=1}^m \mathrm{vol}(U_0) \int\limits_{\bfM_\mathbb{Y}(\AA_\fin) } \omega(h) \omega(k_j)\phi_{\mathrm{fin}}(b_1\otimes w_- +v_0\otimes w_+)\overline{a_f(\langle v_0, v_0\rangle)(h k_j)} dh  
\end{split}
\end{equation*}
Thus $a(v_0;\phi;f)$ is a finite sum of integrals of the form
\begin{equation}
\label{eq-Fourier-coefficient-nondegenerate-2}
    \int\limits_{\bfM_\mathbb{Y}(\AA_\fin) } \omega(h) \phi_{\mathrm{fin}}(b_1\otimes w_- +v_0\otimes w_+)\overline{a_f(\langle v_0, v_0\rangle)(h)} dh.
\end{equation}
It suffices to show that \eqref{eq-Fourier-coefficient-nondegenerate-2} is a finite sum.

Write $h\in \bfM_\mathbb{Y}(\AA_\fin)$ as $h=\diag(a, \overline{a}^{-1})$ with $a\in \AA_{E,\fin}^\times$. Let $U_{\bfM_\mathbb{Y}}=U\cap \bfM_\mathbb{Y}(\AA_\fin)$. We claim that the set of $\diag(a, \overline{a}^{-1})$ with 
\begin{equation*}
    \omega\left( \begin{pmatrix} a & \\ & \overline{a}^{-1}\end{pmatrix}\right) \phi_\fin(b_1\otimes w_- +v_0\otimes w_+) =\chi(a)^{2+n} |a|_{\AA}^{\frac{n}{2}} \phi_\fin(b_1\otimes w_- \overline{a}^{-1}+v_0a\otimes w_+) \not=0
\end{equation*}
has a finite number of $U_{\bfM_\mathbb{Y}}$ cosets. Without loss of generality we may assume that
\begin{equation*}
    \phi_\fin(b_1\otimes w_- \overline{a}^{-1}+v_0a\otimes w_+)=\phi_1(b_1\otimes w_- \overline{a}^{-1}) \phi_2(v_0a\otimes w_+),
\end{equation*}
where $\phi_1$ and $\phi_2$ are characteristic functions on $(b_1\otimes \mathbb{Y})(\AA_\fin)$ and $(\bfV_0\otimes w_+)(\AA_\fin)$ respectively. These conditions on $\phi_1$ and $\phi_2$ imply that there is some $t\in \GL_1(\AA_{E,\fin})\cap \hat{\mathcal{O}_E}$ such that $\overline{a}^{-1}\in t^{-1} \hat{\mathcal{O}_E}$ and $a\in t^{-1} \hat{\mathcal{O}_E}$, where $\hat{\mathcal{O}_E}$ is the maximal compact subgroup of $\AA_{E,\fin}$. Hence, there is a finite number of $U_{\bfM_\mathbb{Y}}$ cosets. This completes the proof.  
\end{proof}

\goodbreak 

\printbibliography
\end{document}